\documentclass[preprint,12pt]{elsarticle}

\usepackage{hyperref}

\usepackage{graphicx}
\usepackage{svg}
\usepackage{caption}
\usepackage{subcaption}

\usepackage{amssymb}
\usepackage{amsmath}
\usepackage{amsthm}
\usepackage{bm}
\newtheorem{theorem}{Theorem}
\newtheorem{lemma}{Lemma}

\journal{}

\begin{document}

\begin{frontmatter}



\title{Linearly Stable Generalizations of ESFR Schemes}


\author[label1]{Mathias Dufresne-Piché\corref{cor1}}
\ead{mathias.dufresne-piche@mail.mcgill.ca}
\cortext[cor1]{Corresponding author.}
\author[label1]{Siva Nadarajah}

\affiliation[label1]{organization={Department of Mechanical Engineering, McGill University},
            city={Montréal},
            postcode={H3A 0G4}, 
            country={Canada}}

\begin{abstract}

The energy stable flux reconstruction (ESFR) method encompasses an infinite family of high-order, linearly stable schemes and thus provides a flexible and efficient framework for achieving high levels of accuracy on unstructured grids. One remarkable property of ESFR schemes is their ability to be expressed equivalently as linearly filtered discontinuous Galerkin (FDG) schemes. In this study, we introduce Sobolev Stable discontinuous Galerkin (SSDG) schemes, a new conservative and linearly stable generalization of ESFR schemes via the FDG framework. Additionally, we review existing generalizations of the ESFR method and consider their relationship with the FDG framework. The linear properties of SSDG schemes are studied via Von Neumann analysis and compared to those of the existing extended ESFR (EESFR) method. It is found that while SSDG and EESFR schemes exhibit fundamentally different dispersive and dissipative behaviors, they can achieve a similar increase in CFL limit and exhibit a similar spectral order of accuracy. Moreover, it is seen that the range of scheme parameters over which SSDG schemes can be used to increase the explicit time-stepping limit is much larger than for EESFR schemes. Finally, it is observed that the order of accuracy of EESFR schemes under $h$-refinement is generally $p+1$ while that of SSDG schemes is $p$.
\end{abstract}


\end{frontmatter}



\section{Introduction}
In the past decades, computational fluid dynamics (CFD) has proven to be a valuable and cost-effective tool to solve complex fluid dynamics problems which arise in a wide range of engineering systems. High-order methods are particularly interesting in that regard as they have the ability to compute more accurate solutions on coarse meshes, thus improving overall computational efficiency.

As a result, in recent years, extensive work has been performed in the research community to characterize and develop high-order methods and improve their associated robustness. Flux reconstruction (FR) schemes are a class of high-order methods which have been introduced by \cite{huynh2007flux} and further extended through the work of \cite{WANG20098161,vincent2011new}. The FR approach is particularly appealing because it allows for the unification of various well-known schemes under one single elegant framework. In particular, collocation-based nodal discontinuous Galerkin (DG) methods and some spectral difference methods can be understood as specific cases of FR \cite{huynh2007flux}. The FR framework has moreover been used by \cite{vincent2011new} to identify an infinite family of linearly stable methods known as energy stable flux reconstruction (ESFR) schemes. The linear properties of ESFR schemes have been investigated by \cite{vincent2011insights, lambert2023l2}, and the ESFR framework was extended to triangular and tetrahedral elements by \cite{williams2013energy, castonguay2012new} and to the advection-diffusion problem by \cite{quaegebeur2020insights, quaegebeur2019stability1, quaegebeur2019stability2}. Through the work of \cite{allaneau2011connections,zwanenburg2016equivalence}, it was moreover shown that ESFR schemes could be expressed as filtered DG (FDG) schemes. More recently, this property has been used by \cite{cicchino2022nonlinearly} to construct non-linearly stable flux reconstruction (NSFR) schemes, an entropy stable extension of the classical ESFR framework.

ESFR schemes are particularly appealing as they are capable of achieving a good trade-off between accuracy and computational cost. More precisely, it is well-known that certain ESFR schemes are associated with larger Courant–Friedrichs–Lewy (CFL) explicit time-step limit than standard nodal DG schemes while being still able to retain a similar order of accuracy \cite{castonguay2012new}. This property is vital when considering realistic viscous CFD test cases, where explicit time-stepping restrictions can become a significant bottleneck.

In recent years, the desirable numerical properties of the schemes arising from the ESFR framework have provided both theoretical and practical motivations for the development of generalizations of the latter. Such generalizations have been proposed by \cite{vincent2015extended} and \cite{trojak2019GSFR}, who introduced the extended ESFR (EESFR) and the generalized Sobolev stable flux reconstruction (GSFR) approaches, respectively. Although EESFR and GSFR schemes share some numerical properties with ESFR schemes, neither is readily compatible with the NSFR framework introduced by \cite{cicchino2022nonlinearly}. Additionally, the non-trivial extension of these schemes to triangular and tetrahedral elements is yet to be proposed.

In this study, we introduce Sobolev stable DG (SSDG) schemes, a new FDG generalization of ESFR schemes. As a result of their FDG structure, SSDG schemes can easily be extended to both triangular and tetrahedral elements and to the NSFR framework. Additionally, we compare the linear properties of SSDG schemes to those of existing generalizations of ESFR schemes and investigate the relationship of the latter with the FDG framework.

This paper is organized as follows. In section \ref{sec:preliminaries}, we first review filtered nodal DG and FR methods in the context on one-dimensional scalar conservation laws and discuss the relationship of the latter with the ESFR framework. In section \ref{sec:gen}, we then introduce EESFR and SSDG schemes, two linearly stable and conservative generalizations of ESFR schemes. Section \ref{sec:VonNeumann} reviews the application of Von Neumann analysis to these two families of numerical schemes. This is finally used in \ref{sec:lin_prop}, where we compare the linear properties of EESFR and SSDG schemes.

\section{Preliminaries}
\label{sec:preliminaries}
Consider the general 1D scalar conservation law given by
\begin{equation}
    \frac{\partial u}{\partial t} + \frac{\partial f(u)}{\partial x} = 0
    \quad ; \quad
    u(x,0) = u_0(x)
    \quad ; \quad
    x \in \Omega,
    \label{eq:ch2_conservation_law}
\end{equation}
where $\Omega$ is an interval in $\mathbb{R}$, $f(u)$ is the flux of $u$ in the $x$ direction and $u(x,t)$ is the conserved scalar. In this section, we start by describing two common high-order spatial discretizations for Eq.(\ref{eq:ch2_conservation_law}) and revisit how the energy stable flux reconstruction (ESFR) schemes developed by \cite{vincent2011new} can be expressed within these frameworks.
\subsection{Filtered Nodal Discontinuous Galerkin}
The DG spatial discretization focuses on finding an approximate weak solution to Eq.(\ref{eq:ch2_conservation_law}). To achieve this, $\Omega$ is first discretized into $N$ non-overlapping intervals $\Omega_k := [x_k,x_{k+1}]$ such that
\begin{equation}
    \Omega = \bigcup_{k=0}^{N-1} \Omega_k
    \quad ; \quad
    \bigcap_{k=0}^{N-1} \Omega_k = \emptyset.
    \label{eq:ch2_discretization}
\end{equation}
To simplify implementation, each element $\Omega_k$ is usually transformed to a standard element $\bar{\Omega} := [-1,1]$ via the linear mappings $\Gamma_k : \Omega_k \to [-1, 1]$ defined as
\begin{equation}
    \xi = \Gamma_k(x) = 2\frac{x-x_k}{x_{k+1}-x_k} -1,
\end{equation}
which have the inverse
\begin{equation}
    x=\Gamma^{-1}_k(\xi) = \frac{1-\xi}{2}x_k + \frac{1+\xi}{2}x_{k+1}.
\end{equation}
This allows rewriting Eq.(\ref{eq:ch2_conservation_law}) as
\begin{equation}
    \frac{\partial \hat{u}_k}{\partial t} + \frac{\partial \hat{f}_k}{\partial \xi} = 0
    \label{eq:ch2_conservation_law_ref}
\end{equation}
with local normalized quantities given by
\begin{equation}
    \hat{u}_k(\xi,t) := u(\Gamma_k^{-1}(\xi), t)
    \quad ; \quad
    \hat{f}_k(u) = \frac{2}{|\Omega_k|}f(u),
\end{equation}
where $|\Omega_k|:= x_{k+1} - x_k$. In the DG framework, the local solutions $\hat{u}_k$ are then approximated using a combination of $p+1$ linearly independent polynomials $\chi_j \in \mathcal{P}^{p}(\bar{\Omega})$ such that
\begin{equation}
    \hat{u}_k(\xi,t) \approx \hat{u}_k^h(\xi,t) = \sum_{j=0}^{p} \hat{u}^h_{jk}(t) \chi_j(\xi) \label{eq:ch2_local_approx}
\end{equation}
and the approximate global solution on $\Omega$ is consequently given by
\begin{equation}
    u(x,t) \approx u^h(x,t) = \bigoplus_{k=0}^{N-1} \hat{u}_k^h(\Gamma_k(x),t).
\end{equation}
A DG discretization is obtained by requiring the residual of Eq.(\ref{eq:ch2_local_approx}) under Eq.(\ref{eq:ch2_conservation_law_ref}) to be orthogonal to $\mathcal{P}^{p}(\bar{\Omega})$. Integrating by parts twice and introducing the numerical flux $f^{*}_k$ yields the strong form of the DG method,
\begin{equation}
    \sum_{j=0}^{N-1}\int_{\bar{\Omega}} \chi_i \chi_j d \xi \frac{d \hat{u}^h_{jk}}{dt}
   +\int_{\bar{\Omega}} \frac{\partial\hat{f}_k}{\partial \xi} \chi_i d\xi = \left[(\hat{f}_k - \hat{f}_k^{*})\chi_i\right]_{\partial \bar{\Omega}}. \label{eq:ch2_DG_strong}
\end{equation}
If the local flux is approximated in the polynomial basis via a collocation projection of $\hat{f}_k(\hat{u}^h_k)$
\begin{equation}
    \hat{f}_k(\xi,t) \approx
    \hat{f}_k^h(\xi,t)=
    \sum_{j=0}^{p} \hat{f}_{jk}^h(t) \chi_j(\xi),
\end{equation}
Eq.(\ref{eq:ch2_DG_strong}) can be rewritten in matrix form as
\begin{equation}
    \mathbf{M} \frac{d\hat{\mathbf{u}}_k^h}{dt}
    = -\mathbf{M} \mathbf{D} \hat{\mathbf{f}}_k^h
    + \left[(\hat{f}_k^h - \hat{f}_k^{*})\bm{\chi}\right]_{\partial \bar{\Omega}},
    \label{eq:DG_strong_matrix}
\end{equation}
where
\begin{equation}
    (\mathbf{M})_{ij} := \int_{\bar{\Omega}} \chi_i \chi_j d \xi
    \quad ; \quad
    \bm{\chi}_i(\xi) = \chi_i(\xi)
    \quad ; \quad
    (\hat{\mathbf{u}}_k^h)_{i} := \hat{u}_{ik}^h
    \quad ; \quad
    (\hat{\mathbf{f}}_k^h)_{i} := \hat{f}_{ik}^h,
    \label{eq:ch2_DG_operators}
\end{equation}
$\mathbf{D}$ is the matrix representation of the differentiation operator on $\mathcal{P}^{p}(\bar{\Omega})$ expressed with respect to the basis $\{\chi_0, \dots, \chi_p\}$ and $\partial \bar{\Omega} := \{-1,1\}$.

The nodal DG spatial discretization introduced in Eq.(\ref{eq:DG_strong_matrix}) can be generalized by applying a linear filter to the right-hand side of Eq.(\ref{eq:DG_strong_matrix}). The resulting family of schemes, which will henceforth be referred to as filtered DG (FDG) schemes, can be written as
\begin{equation}
    (\mathbf{M}+\mathbf{K}) \frac{d\hat{\mathbf{u}}_k^h}{dt}
    = -\mathbf{M} \mathbf{D} \hat{\mathbf{f}}_k^h
    + \left[(\hat{f}_k^h - \hat{f}_k^{*})\bm{\chi}\right]_{\partial \bar{\Omega}},
    \label{eq:FDG_matrix}
\end{equation}
where $\mathbf{K} \in \mathbb{R}^{(p+1) \times (p+1)}$ defines the filtering operator employed. When the flux function is linear, \textit{ie.}, $f(u)=au$, Eq.(\ref{eq:FDG_matrix}) takes the form
\begin{equation}
    (\mathbf{M}+\mathbf{K}) \frac{d\hat{\mathbf{u}}_k^h}{dt}
    = -\hat{a}_k\mathbf{M} \mathbf{D} \hat{\mathbf{u}}_k^h
    + \left[(\hat{a}_k\hat{u}^h_k - \hat{f}_k^{*})\bm{\chi}\right]_{\partial \bar{\Omega}},
    \label{eq:FDG_matrix_linear}
\end{equation}
where $\hat{a}_k := 2a / |\Omega_k|$ is the normalized local advection velocity.

\subsubsection{Linear Stability and Conservativeness}
\label{sec:FDG_theo}
Simple sufficient conditions for linear stability and conservativeness of FDG schemes have been provided by \cite{allaneau2011connections}. They are summarized here for completeness.
\begin{theorem}
\label{theo:FDG_stability}
    The spatial discretization given by Eq.(\ref{eq:FDG_matrix_linear}) is stable provided that $\mathbf{M} + \mathbf{K} \succ 0$ and the numerical flux is given by
    \begin{equation}
        \hat{f}^*(u^{+},u^{-}) = \hat{a} \frac{u^{+}+u^{-}}{2} - (1-\alpha)|\hat{a}| \frac{u^{+}-u^{-}}{2},
        \label{eq:num_flux}
    \end{equation}
    where $u^+$ and $u^-$ represent the solution values at the right and left of a given elemental interface and $\alpha \in [0,1]$. 
\end{theorem}
\begin{proof}
    If $\mathbf{M} + \mathbf{K} \succ 0$, then
    \begin{equation}
        ||u^h||_{\mathbf{M}+\mathbf{K}}^2 :=
        \sum_{k=0}^N
        (\hat{\mathbf{u}}_k^h)^T(\mathbf{M}+\mathbf{K}) \hat{\mathbf{u}}_k^h
    \end{equation}
    is a norm of the solution $u^h$. Pre-multiplying Eq.(\ref{eq:FDG_matrix_linear}) by $(\hat{\mathbf{u}}_k^h)^T$ and summing over all elements yields
    \begin{align}
    &\frac{1}{2} \frac{d}{dt} ||u^h||_{\mathbf{M}+\mathbf{K}}^2
    = \sum_{k=0}^N-\hat{a}_k(\hat{\mathbf{u}}_k^h)^T\mathbf{M} \mathbf{D} \hat{\mathbf{u}}_k^h
    + (\hat{\mathbf{u}}_k^h)^T\left[(\hat{a}_k\hat{u}^h_k - \hat{f}_k^{*})\bm{\chi}\right]_{\partial \bar{\Omega}}.
    \label{eq:stablity_derivation}
    \end{align}
    It is well known from the derivation of the linear stability of the DG method that the right-hand side of Eq.(\ref{eq:stablity_derivation}) is non-positive for numerical fluxes given by Eq.(\ref{eq:num_flux}) (see \cite{hesthaven2007nodal, allaneau2011connections} for example). Hence,
    \begin{equation}
        \frac{d}{dt}||u^h||_{\mathbf{M}+\mathbf{K}}^2 \leq 0
    \end{equation}
    and boundedness of all other norms of the solution follows from the equivalence of norms on finite-dimensional vector spaces.
\end{proof}

\begin{theorem}
    \label{theo:FDG_conservative}
    Let $(\hat{\mathbf{e}}^n)_i : = \delta_{in}$, where $\delta_{in}$ stands for the Kronecker delta. The spatial discretization given by Eq.(\ref{eq:FDG_matrix_linear}) is conservative provided that $(\hat{\mathbf{e}}^0)^T\mathbf{K} = 0$ when $\mathbf{K}$ is expressed in the Legendre basis.
\end{theorem}
\begin{proof}
    Without loss of generality, assume that Eq.(\ref{eq:FDG_matrix}) is written in the Legendre basis. Letting $\psi_n$ denote the $n$th order Legendre polynomial, we first note that
    \begin{equation}
    \frac{d}{dt} \int_{\bar{\Omega}} \hat{u}^h_k d \xi
    =
    \frac{d}{dt} \int_{\bar{\Omega}} \psi_{0} \hat{u}^h_k d \xi
    = \frac{d}{dt} \mathbf{e}_0^T \mathbf{M} \hat{\mathbf{u}}_k^h
    =\frac{d}{dt}\mathbf{e}_0^T(\mathbf{M} + \mathbf{K})\hat{\mathbf{u}}^h_k.
\end{equation}
Pre-multiplying Eq.(\ref{eq:FDG_matrix}) by $(\mathbf{e}^0)^T$ yields
\begin{align}
    \frac{d}{dt} \int_{\bar{\Omega}} \hat{u}^h_k d \xi
    &=
    \frac{d}{dt}(\mathbf{e}^0)^T(\mathbf{M} + \mathbf{K})\hat{\mathbf{u}}^h_k
    =
    -(\mathbf{e}^0)^T\mathbf{M}\mathbf{D} \hat{\mathbf{f}}_k^h
    + (\mathbf{e}^0)^T\left[(\hat{f}_k^h - \hat{f}_k^{*})\bm{\chi}\right]_{\partial \bar{\Omega}}
    \nonumber \\
    &=
    -\int_{\bar{\Omega}} \frac{\partial \hat{f}^h_k}{\partial \xi} \psi_0 d\xi
    + \left[(\hat{f}_k^h-\hat{f}^*_k)\psi_0\right]_{\partial\bar{\Omega}}
    =
    -\left[\hat{f}_k^h\right]_{\partial\bar{\Omega}}
    + \left[(\hat{f}_k^h-\hat{f}^*_k)\right]_{\partial\bar{\Omega}} \nonumber \\
    &= -\left[\hat{f}^*_k\right]_{\partial\bar{\Omega}},
\end{align}
which completes the proof.
\end{proof}

\subsection{Flux Reconstruction}
The flux reconstruction (FR) framework was initially introduced by \cite{huynh2007flux} as a means to develop a family of spatial discretizations for Eq.(\ref{eq:ch2_conservation_law}). FR schemes are usually derived by constructing piecewise high-order polynomials satisfying Eq.(\ref{eq:ch2_conservation_law}) in its differential form and introducing correction functions to ensure continuity of the flux function at the elemental interfaces.

In this paper, for notational convenience, we, however, somewhat depart from the usual approach and define FR schemes via
\begin{equation}
    \frac{d\hat{\mathbf{u}}_k^h}{dt}
    = -\mathbf{D} \hat{\mathbf{f}}_k^h
    + (\mathbf{M} + \mathbf{Q})^{-1}\left[(\hat{f}_k^h - \hat{f}_k^{*})\bm{\chi}\right]_{\partial \bar{\Omega}}.
    \label{eq:FR_unconventional}
\end{equation}
It can readily be seen that the schemes described by Eq.(\ref{eq:FR_unconventional}) are effectively obtained by applying a linear filter to the discontinuous flux terms in the DG formulation. Using the notation in Eq.(\ref{eq:FR_unconventional}), the derivatives of the correction functions $g_R$ and $g_L$ typically used in the FR framework are given by
\begin{align}
    g_R'(\xi) &= (\bm{\chi}(\xi))^T(\mathbf{M} + \mathbf{Q})^{-1} \mathbf{r}, \\
    g_L'(\xi) &= -(\bm{\chi}(\xi))^T(\mathbf{M} + \mathbf{Q})^{-1} \mathbf{l},
\end{align}
where $\mathbf{r} := \bm{\chi}(1)$ and $\mathbf{l} : = \bm{\chi}(-1)$. To ensure Eq.(\ref{eq:FR_unconventional}) is strictly equivalent to the usual FR formulation, one must require $\mathbf{Q} \in \mathbb{R}^{(p+1)\times(p+1)}$ to be such that the conventional boundary conditions on $g_R$ and $g_L$ are satisfied, \textit{ie.},
\begin{align}
    g_R(1) &= 1 \quad ; \quad g_R(-1) = 0 \implies \int_{-1}^{1} g_R'(\xi) d\xi = 1 \\
    g_L(1) &= 0 \quad ; \quad g_L(-1) = 1 \implies \int_{-1}^{1} g_L'(\xi) d\xi = -1.
\end{align}
This will be the case provided that $\mathbf{Q}$ is such that the spatial discretization shown in Eq.(\ref{eq:FR_unconventional}) is conservative.

\subsubsection{Conservativeness}
Following the simple argument presented by \cite{vincent2015extended}, a sufficient condition for conservativeness of the schemes given by Eq.(\ref{eq:FR_unconventional}) can be derived.
\begin{theorem}
\label{theo:FR_conservative}
The spatial discretization given by Eq.(\ref{eq:FR_unconventional}) is conservative provided that $(\hat{\mathbf{e}}^0)^T\mathbf{Q} = \mathbf{Q}\hat{\mathbf{e}}^0 =0$ when $\mathbf{Q}$ is expressed in the Legendre basis.
\end{theorem}

\begin{proof}
 Without loss of generality, assume that Eq.(\ref{eq:FDG_matrix}) is written in the Legendre basis. Then,
 \begin{align}
     \int_{-1}^{1} g_R'(\xi) d\xi &= \left(\int_{-1}^{1} (\bm{\psi}(\xi))^Td\xi\right)(\mathbf{M} + \mathbf{Q})^{-1} \mathbf{r}, \\
    \int_{-1}^{1} g_L'(\xi) d\xi &= -\left(\int_{-1}^{1} (\bm{\psi}(\xi))^Td\xi\right)(\mathbf{M} + \mathbf{Q})^{-1} \mathbf{l},
 \end{align}
 where $\bm{\psi}_i(\xi) = \psi_i(\xi)$. Leveraging the properties of $\mathbf{M}$ and $\mathbf{Q}$ in the Legendre basis, one concludes
 \begin{align}
     \int_{-1}^{1} g_R'(\xi) d\xi &=
     2(\hat{\mathbf{e}}^0)^T (\mathbf{M} + \mathbf{Q})^{-1} \mathbf{r} = 2(\mathbf{M}^{-1})_{00} =1, \\
     \int_{-1}^{1} g_L'(\xi) d\xi &=
     -2(\hat{\mathbf{e}}^0)^T (\mathbf{M} + \mathbf{Q})^{-1} \mathbf{l} = -2(\mathbf{M}^{-1})_{00} = -1,
 \end{align}
 which is compatible with the boundary conditions prescribed on $g_R$ and $g_L$. Hence, provided that the sufficient condition is satisfied, Eq.(\ref{eq:FR_unconventional}) is equivalent to the standard FR formulation and yields a conservative family of spatial discretizations.
\end{proof}

By linear independence of $\bm{\psi}(1)$ and $\bm{\psi}(-1)$, it follows that any pair of correction functions satisfying the required boundary conditions can be generated through the selection of a suitable matrix $\mathbf{Q}$ of the form prescribed above. Moreover, if $g_R$ and $g_L$ are symmetric, such a symmetric $\mathbf{Q}$ can always be found.

\subsection{Energy Stable Flux Reconstruction Schemes}
ESFR schemes are a family of linearly stable FR schemes developed by \cite{vincent2011new}. As shown by \cite{allaneau2011connections} and later generalized by \cite{zwanenburg2016equivalence}, ESFR schemes can be constructed via both the FDG and FR approaches and can hence be defined through
\begin{equation}
    (\mathbf{K})_{ESFR} = \mathbf{Q}_{ESFR} = \frac{1}{2} c_p(\mathbf{D}^p)^T \mathbf{M} \mathbf{D}^p,
    \label{eq:ESFR_definition}
\end{equation}
where $\mathbf{D}^p$ is used to denote the $p$th power of the differentiation matrix. The consistency of this statement relies on the fact $(\mathbf{K})_{ESFR}\mathbf{D}=0$, which ensures equivalence of the FDG and FR definitions. ESFR schemes are provably linearly stable ($\mathbf{M} + (\mathbf{K})_{ESFR} \succ 0$) whenever $c_p > c^{-}_{ESFR}$ for
\begin{equation}
    c_{ESFR}^{-} = \frac{-2}{(2p+1)(k_p)^2}, \quad
    \text{where} \quad
    k_p := \frac{(2p)!}{2^p p!}.
\end{equation}
Other important values of $c_p$ have also been identified by \cite{vincent2011new} and are listed here for completeness
\begin{equation}
    c_{SD} := \frac{2p}{(2p+1)(p+1)(k_p)^2}
    \quad ; \quad
    c_{HU} := \frac{2(p+1)}{(2p+1)p(k_p)^2},
    \label{eq:ch2_c_values}
\end{equation}
where using $c_{SD}$ and $c_{HU}$ results in a spectral difference scheme and Huynh's $g_2$ scheme respectively. Additionally, for notational convenience, in what follows, we will denote the ESFR scheme maximizing the CFL time-step limit as $c_{ESFR}^{\tau}$.

\section{Generalizations of ESFR Schemes}
\label{sec:gen}
We now present the generalizations of ESFR schemes studied in this paper.

\subsection{Extended Energy Stable Flux Reconstruction Schemes}
Extended ESFR (EESFR) schemes have been introduced by \cite{vincent2015extended} as an extended family of linearly stable FR schemes which generalizes the ESFR framework. EESFR schemes are defined by requiring $\mathbf{Q}$ to satisfy Theorem \ref{theo:FR_conservative} along with
\begin{align}
    &\mathbf{Q} = \mathbf{Q}^T,
    \label{eq:ch6_Q_req1}\\
    &\mathbf{Q}{}\mathbf{D} + \mathbf{D}^T\mathbf{Q}^T = 0,
    \label{eq:ch6_Q_req2}\\
    &\mathbf{M} + \mathbf{Q} \succ 0,
    \label{eq:ch6_Q_req3} \\
    & \mathbf{J} \mathbf{Q} = \mathbf{Q} \mathbf{J},
    \label{eq:Q_req4}
\end{align}
where $\mathbf{J}_{ij} := \delta_{ij}(-1)^{i+1}$. As shown by \cite{vincent2015extended}, Eq.(\ref{eq:ch6_Q_req1}), Eq.(\ref{eq:ch6_Q_req2}) and Eq.(\ref{eq:ch6_Q_req3}) ensure stability of the FR spatial discretization while Eq.(\ref{eq:Q_req4}) guarantees symmetry of the correction functions. The general form of $\mathbf{Q}_{EESFR}$ for $p \in \{3,4,5,6\}$ has been derived by \cite{vincent2015extended}. While the number of scalar parameters required to describe $\mathbf{Q}_{EESFR}$ increases with $p$, in this paper, for simplicity, as done by \cite{vermeire2016properties} we will restrict ourselves to two-parameter EESFR schemes for which $\mathbf{Q}_{EESFR}$, in the Legendre basis, is given by
\begin{equation}
    \mathbf{Q}_{EESFR}
    =
    \begin{bmatrix}
        0 & \cdots & 0 & 0 &0 \\
        \vdots & \ddots & \vdots & \vdots & \vdots \\
        0 & \cdots & 0 & 0 & -\beta q_1 \\
        0 & \cdots & 0 & q_1 & 0 \\
        0 & \cdots & -\beta q_1 & 0 & q_0
    \end{bmatrix},
    \quad
    \text{with }
    \beta = \frac{2p-1}{2p - 3}.
    \label{eq:ch6_Q}
\end{equation}
In this case, it can easily be checked that positive-definiteness of $\mathbf{M} + \mathbf{Q}_{EESFR}$ is guaranteed provided that
\begin{align}
    q_{1} >& -\frac{2}{2p-1}, \label{eq:ch6_stability_req1} \\
    q_1^2 <& \frac{2(2p-3)}{(2p-1)^2}\left(\frac{2}{2p+1}+q_0\right) \label{eq:ch6_stability_req2},
\end{align}
which is consistent with the stability conditions presented by \cite{vincent2015extended} and \cite{vermeire2016properties} although written in a slightly different form. In what follows, we will denote the stability upper and lower bounds for $q_1$ by $q_1^+$ and $q_1^{-}$. Both $q_1^+$ and $q_1^-$ are functions of $q_0$. When $q_1=0$, the one-parameter family of ESFR schemes is recovered. In this case, the ESFR and EESFR representations are connected via
\begin{equation}
    q_0 = k_p^2 c_p.
    \label{eq:param_scale}
\end{equation}
Finally, as observed by \cite{vermeire2016properties}, it can readily be seen from Eq.(\ref{eq:ch6_Q}) that in the limit $q_0 \to \infty$, EESFR schemes of degree $p$ collapse to ESFR schemes of degree $p-1$.

\subsection{Sobolev Stable Discontinuous Galerkin Schemes}
ESFR schemes are derived by requiring the boundedness of a broken Sobolev norm of the numerical solution. In this paper, we extend this idea by introducing a family of FDG discretizations which ensure the boundedness of more general Sobolev norms. These schemes, which will hereafter be referred to as Sobolev stable DG (SSDG), are defined via
\begin{equation}
    (\mathbf{K})_{SSDG} := \frac{1}{2} \sum_{k=1}^p c_k(\mathbf{D}^k)^T \mathbf{M} \mathbf{D}^k.
    \label{eq:SSDG}
\end{equation}
It is straightforward to verify that $(\mathbf{K})_{SSDG}$ always satisfies the sufficient condition for conservativeness and ensures linear stability of SSDG schemes whenever $\forall_{k=1}^{p} c_k > 0$. This follows directly from Theorems \ref{theo:FDG_stability} and \ref{theo:FDG_conservative}  presented in section \ref{sec:FDG_theo}. Clearly, when $\forall_{k=1}^{p-1}c_k=0$, one recovers $(\mathbf{K})_{ESFR}$. In this paper, for simplicity, we will mainly restrict ourselves to two-parameter SSDG schemes, that is, SSDG schemes for which $\forall_{k=1}^{p-2}c_k = 0$. In this case, since $(\mathbf{K})_{SSDG}$ is always diagonal in the Legendre basis, it is found that
\begin{equation}
    (\mathbf{K})_{SSDG} =
    \text{diag}\left(0,...,0, k_{p-1}^2c_{p-1}, \frac{1}{3}k_{p}^2c_{p-1} + k_{p}^2c_p\right).
    \label{eq:ch5_K}
\end{equation}
from which, by simple application of Theorem \ref{theo:FDG_stability}, the following sharper stability requirement can be deduced
\begin{align}
    c_{p-1} &> \frac{-2}{k_{p-1}^2(2p-1)}, \label{eq:ch5_stability_constraint1} \\
    c_p &> -\frac{1}{3}c_{p-1} - \frac{2}{k_p^2(2p+1)}. \label{eq:ch5_stability_constraint2}
\end{align}
In what follows, the stability lower bound for $c_{p-1}$ will be denoted by $c_{p-1}^-$. As can be seen from Eq.(\ref{eq:ch5_stability_constraint1}) and Eq.(\ref{eq:ch5_stability_constraint2}), $c_{p-1}^-$ is a function of $c_p$.

From Eq.(\ref{eq:ch5_K}), it can be seen that in the limit $c_{p-1} \to \infty$, the $p$th and $(p-1)$th order Legendre modes are annihilated from the scheme and SSDG schemes of degree $p$ collapse to a DG scheme of degree $p-2$. This behavior is analogous to that of ESFR schemes, which become a DG scheme of degree $p-1$ in the limit $c_p \to \infty$ \cite{vincent2011new}. Finally, Eq.(\ref{eq:ch5_K}) also shows that as $c_{p} \to \infty$, SSDG schemes of degree $p$ collapse to ESFR schemes of degree $p-1$.

\subsection{Generalized Sobolev Stable Flux Reconstruction}
Generalized Sobolev Stable FR (GSFR) schemes are another generalization of the ESFR framework proposed by \cite{trojak2019GSFR}, which can be defined through
\begin{equation}
    \mathbf{Q}_{GSFR} = \sum_{k=1}^p b_k(\mathbf{D}^k)^T \mathbf{M} \mathbf{D}^k.
\end{equation}
One should note that although the notation used in this definition somewhat differs from that used in the work of \cite{trojak2019GSFR}, it is mathematically equivalent to the latter. While GSFR schemes may be regarded as locally stable in the sense that they satisfy the condition provided by \cite{trojak2019GSFR} Eq.(37), some schemes in this family admit unbounded solutions, which implies they are not linearly stable in the usual sense. For instance, the $p=3$ GSFR scheme with $b_1 = 0.03$, $b_2 = 0.03$, $b_3=0.0075$ and an upwind numerical flux is not linearly stable, although it satisfies the local stability condition of \cite{trojak2019GSFR} Eq.(37). Details are provided in \ref{app2}. Since GSFR schemes cannot be regarded as linearly stable in the usual sense, they will not be further studied as part of this work, and we will restrict ourselves to the characterization of EESFR and SSDG schemes.

\subsection{Connections between Filtered DG and FR Schemes}
As demonstrated in \ref{app1}, it is not difficult to show that for general numerical fluxes given by Eq.(\ref{eq:num_flux}), ESFR schemes are the only FR schemes with symmetric correction functions which can be expressed in the FDG framework. Analogously, as shown in \ref{app1}, ESFR schemes are also the only FDG schemes with a symmetric $\mathbf{K}$ which can be expressed in the FR framework. Hence, the relationship between FDG and FR schemes can be summarized graphically as shown in Figure \ref{fig:Venn_FR_FDG}. This also implies that the intersection between the EESFR and SSDG families of numerical schemes exactly coincides with the set of all ESFR schemes, as shown in Figure \ref{fig:Venn2}. It should be stressed that this holds for general numerical fluxes given by Eq.(\ref{eq:num_flux}); if we restrict ourselves to upwind numerical fluxes in one dimension, then it can be shown that all FR schemes can be expressed as FDG schemes. Details regarding this degenerate case are provided in \ref{app1}.
\begin{figure}
    \centering
    \begin{subfigure}{.49\textwidth}
        \centering
        \includegraphics[width=0.99\linewidth, page = 1]{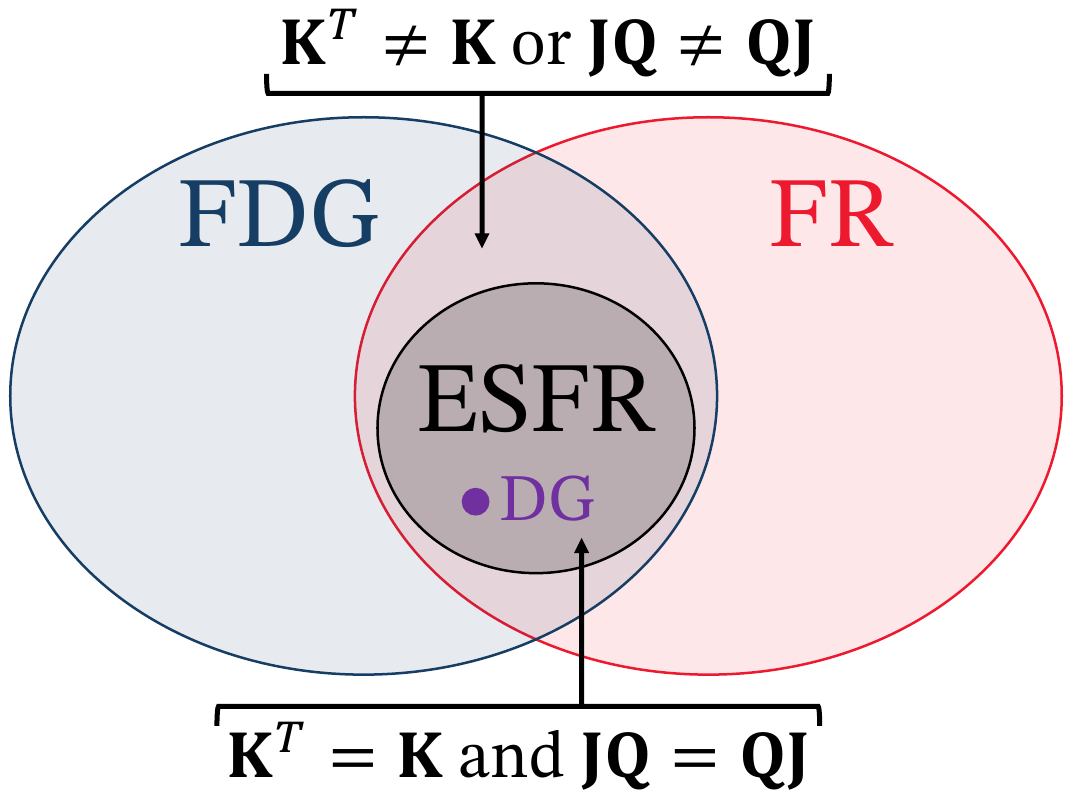}
        \caption{FR and FDG schemes.}
        \label{fig:Venn_FR_FDG}
    \end{subfigure}
    \begin{subfigure}{.49\textwidth}
        \centering
        \includegraphics[width=0.99\linewidth, page = 2]{Images/figures/venn.pdf}
        \caption{EESFR and SSDG schemes.}
        \label{fig:Venn2}
    \end{subfigure}
    \caption{Relationship between FR, FDG, EESFR, SSDG and ESFR schemes.}
\end{figure}

\section{Von Neumann Analysis}
\label{sec:VonNeumann}
To investigate the properties of the generalizations of the ESFR framework presented in the previous section, we follow the approach described by \cite{vincent2011insights, huynh2007flux} and use Von Neumann analysis.

Letting $f(u) := u$ in Eq.(\ref{eq:ch2_conservation_law}) results in the classic linear advection problem
\begin{equation}
    \frac{\partial u}{\partial t} + \frac{\partial u}{\partial x} = 0,
\end{equation}
which admits Bloch wave solutions of the form
\begin{equation}
    u(x,t) = e^{I(\theta x - \omega t)},
\end{equation}
provided that the temporal frequency $\omega(\theta) \in \mathbb{C}$ is related to the wavenumber $\theta \in \mathbb{R}$ via the dispersion relation
\begin{equation}
    \text{Re}(\omega) = \theta
\end{equation}
and the dissipation relation
\begin{equation}
    \text{Im}(\omega) = 0,
\end{equation}
where $I := \sqrt{-1}$. To compute the dispersion-dissipation relations of the schemes previously discussed, for simplicity, following the approach of \cite{vincent2011new}, we take $\forall_{k=0}^{N} |\Omega_k| = 1$, which results in $\hat{f}_k^h = 2\hat{u}^h_k$ in Eq.(\ref{eq:FDG_matrix}) and Eq.(\ref{eq:FR_unconventional}). We then seek numerical Bloch wave solutions of the form
\begin{equation}
    \hat{\mathbf{u}}^h_k = e^{I(k\theta^h-\omega^h)} \hat{\mathbf{v}}.
    \label{eq:Bloch_wave}
\end{equation}
Substituting Eq.(\ref{eq:Bloch_wave}) in Eq.(\ref{eq:FDG_matrix}) and Eq.(\ref{eq:FR_unconventional}) yields
\begin{equation}
    I\omega^h \hat{\mathbf{v}} = 2\mathbf{H} \hat{\mathbf{v}},
    \label{eq:eig}
\end{equation}
where
\begin{align}
    \mathbf{H}_{FDG} &=
    (\mathbf{M}+\mathbf{K})^{-1}\mathbf{M}\mathbf{D}
    - \frac{1}{2}\alpha (\mathbf{M}+\mathbf{K})^{-1}\mathbf{r}\mathbf{r}^T
    + \frac{1}{2}(2-\alpha) (\mathbf{M}+\mathbf{K})^{-1} \mathbf{l}\mathbf{l}^T \nonumber \\
    &+ \frac{1}{2}\alpha (\mathbf{M}+\mathbf{K})^{-1} \mathbf{r} \mathbf{l}^T e^{I\theta^h}
    -\frac{1}{2}(2-\alpha) (\mathbf{M}+\mathbf{K})^{-1} \mathbf{l} \mathbf{r}^T  e^{-I\theta^h},
    \label{eq:ch3_Q_FDG}
\end{align}
for FDG schemes and
\begin{align}
    \mathbf{H}_{FR}
    &=
    \mathbf{D}
    -\frac{1}{2}\alpha (\mathbf{M} + \mathbf{Q})^{-1} \mathbf{r} \mathbf{r}^T
    + \frac{1}{2}(2-\alpha) (\mathbf{M} + \mathbf{Q})^{-1} \mathbf{l} \mathbf{l}^T \nonumber \\
    &+\frac{1}{2}\alpha (\mathbf{M} + \mathbf{Q})^{-1} \mathbf{r} \mathbf{l}^T e^{I\theta^h}
    - \frac{1}{2}(2-\alpha) (\mathbf{M} + \mathbf{Q})^{-1} \mathbf{l} \mathbf{r}^{T} e^{-I\theta^h}
    \label{eq:ch3_Q_FR}
\end{align}
for FR schemes. Eq.(\ref{eq:eig}) is a classical eigenvalue problem and admits $p+1$ eigenpairs $(\omega^h, \hat{\mathbf{v}})$ for every input wavenumber $\theta^h$. To identify the physical eigenpair among the $p+1$ admissible ones, the approaches described by \cite{vincent2011insights,moura2015linear} can be used. As shown by \cite{huynh2007flux}, the spectrum of $\mathbf{H}$ is independent of the polynomial basis selected for the implementation of the numerical scheme. Hence, results derived from Von Neumann analysis can effectively be regarded as basis-free.

\subsection{Spectral Error}
The combined spectral error $E_T$ of a numerical scheme can be calculated from its dispersion-dissipation relations via
\begin{equation}
    E_T(\theta^h) := |\omega^h(\theta^h) - \omega(\theta^h)|,
\end{equation}
where $\omega$ and $\omega^h$ stand for the exact and numerical temporal frequencies, respectively. As done by \cite{vincent2011insights, huynh2007flux}, this can be used to define a spectral order of accuracy $A_T$ through
\begin{equation}
    A_T := \frac{\ln(E_T(\theta_R^h)) - \ln(E_T(\theta^h_R / 2))}{\ln(2)} - 1,
\end{equation}
where $\theta^h_R$ is a sufficiently small wavenumber which ensures $E_T << 1$ and is in the well-resolved range.

\subsection{Stability}
To determine the CFL time-step limit of a numerical scheme via Von Neumann analysis, we first recall that, by construction, its associated numerical Bloch wave solutions satisfy the ordinary differential equation
\begin{equation}
    \frac{d\hat{\mathbf{u}}_k^h}{dt} = -2\mathbf{H} \hat{\mathbf{u}}_k^h.
    \label{eq:ch3_Bloch_ODE}
\end{equation}
Applying an explicit Runge-Kutta (RK) temporal discretization with a time-step $\tau$ to Eq.(\ref{eq:ch3_Bloch_ODE}) leads to a fully discrete scheme of the form
\begin{equation}
    \hat{\mathbf{u}}_k^{h(n+1)} = \mathbf{R} \hat{\mathbf{u}}_k^{h(n)},
    \label{eq:ch3_RK}
\end{equation}
where $\hat{\mathbf{u}}_k^{h(n+1)}$ and $\hat{\mathbf{u}}_k^{h(n)}$ denote the numerical solution at time $t+\tau$ and $t$ respectively and $\mathbf{R} \in \mathbb{C}^{(p+1)\times(p+1)}$ depends on $\tau$ and the type of time-stepping scheme selected. Namely, for a three-stage third-order RK scheme (RK33) $\mathbf{R}$ can be written as \citep{carpenter1994fourth}
\begin{equation}
    \mathbf{R} = \mathbf{I} - 2\tau \mathbf{H} + \frac{1}{2!} (2\tau\mathbf{H})^2 - \frac{1}{3!}(2\tau \mathbf{H})^3,
    \label{eq:ch3_RK33}
\end{equation}
for a four-stage fourth-order RK scheme (RK44), $\mathbf{R}$ has the form \citep{carpenter1994fourth}
\begin{equation}
    \mathbf{R} = \mathbf{I} - 2\tau \mathbf{H} + \frac{1}{2!} (2\tau\mathbf{H})^2 - \frac{1}{3!}(2\tau \mathbf{H})^3 + \frac{1}{4!}(2\tau \mathbf{H})^4,
    \label{eq:ch3_RK44}
\end{equation}
and for the ``optimal'' five-stage fourth order low-storage RK scheme (RK45) identified by \cite{carpenter1994fourth}, $\mathbf{R}$ has the form
\begin{equation}
    \mathbf{R} = \mathbf{I} - 2\tau \mathbf{H} + \frac{1}{2!} (2\tau\mathbf{H})^2 - \frac{1}{3!}(2\tau \mathbf{H})^3 + \frac{1}{4!}(2\tau \mathbf{H})^4 - \frac{1}{200}(2\tau \mathbf{H})^5.
    \label{eq:ch3_RK45}
\end{equation}
The stability of the fully discrete scheme described by Eq.(\ref{eq:ch3_RK}) is guaranteed if $\rho(\mathbf{R})$, the spectral radius of $\mathbf{R}$, is less than unity for all possible values of $\mathbf{H}$. Thus, we define $\tau_{CFL}$ as the maximum value of $\tau$ for which $\rho(\mathbf{H}) < 1$ for all wavenumbers $-\pi \leq \theta^h \leq \pi$. Since it was assumed that $|\Omega_k|=1$ and $a=1$ when constructing $\mathbf{H}$, $\tau_{CFL}$ can effectively be regarded as the maximum non-dimensional time-step for the scheme.

\section{Results and Discussion}
\label{sec:lin_prop}
This section focuses on characterizing the linear properties of EESFR and SSDG schemes. Most results presented for EESFR schemes have already been discussed in depth by \cite{vermeire2016properties} and are hence presented for comparison purposes. For concision, the sections that follow mostly focus on the case $p=3$ with an upwind numerical flux. Schemes of higher order are expected to behave in a qualitatively similar way.

\subsection{Dispersion-Dissipation Relations}
The dispersion-dissipation relations for EESFR and SSDG schemes for $c_p \in \{c_{DG}, c_{SD}, c_{HU}\}$ and $q_0 \in \{k_p^2c_{DG}, k_p^2c_{SD}, k_p^2c_{HU}\}$ are shown in Figures \ref{fig:ch5_DG_dispers} through \ref{fig:ch5_HU_dissip}. For EESFR schemes, in each case, $q_1$ is varied from $0.5 q_1^-$ to $0.5 q_1^+$ while for SSDG schemes, $c_{p-1}$ is varied between $0.5c^{-}_{p-1}$ and $3\times10^{-2}$. In all cases, the exact dispersion-dissipation relation for the linear advection equation is indicated with a gray dashed line, while the dispersion-dissipation profiles for the case $c_{p-1} =0$ or $q_1 =0$ are shown with a loosely dotted black line. Moreover, for SSDG schemes, the dispersion-dissipation profiles obtained in the limit $c_{p-1} \to \infty$ are displayed using a tightly dotted black line.

From Figures \ref{fig:ch5_DG_dispers}, \ref{fig:ch5_SD_dispers}, and \ref{fig:ch5_HU_dispers}, it is clear that the dispersive behavior of EESFR schemes differs significantly from that of SSDG schemes. For EESFR schemes, positive values of $q_1$ lead to an increase in dispersion at high wavenumbers and a comparatively small decrease in dispersion at low wavenumbers, leaving the range over which the dispersion profile is accurate mostly unchanged compared to the case $q_1 = 0$. Conversely, negative values of $q_1$ reduce dispersion at high wavenumbers and increase dispersion at lower wavenumbers, thereby decreasing the range over which the dispersion profile of the scheme is accurate when compared to the case $q_1 = 0$. For strongly negative values of $q_1$, the dispersion is close to 0 at high wavenumbers, thus resulting in schemes which are not capable of advecting initial conditions with a high wavenumber content. For SSDG schemes, the dispersive behavior observed appears to be much closer qualitatively to that of ESFR schemes as initially described in the work of \cite{vincent2011insights}. Namely, increasing $c_{p-1}$ decreases dispersion at high wavenumbers. As expected, when $c_{p-1} \to \infty$ the dispersion vanishes over wavenumbers in the range $(p-1)\pi \leq \theta \leq (p+1)\pi$, thereby preventing the correct advection of initial conditions with a wavenumber content in this range.

As can be observed in \ref{fig:ch5_DG_dissip}, \ref{fig:ch5_SD_dissip}, and \ref{fig:ch5_HU_dissip}, EESFR and SSDG schemes also exhibit a fundamentally different dissipative behavior. For EESFR schemes, positive values of $q_1$ increase dissipation at high wavenumbers while increasing the range over which the dissipation relation of the scheme is accurate. Negative values of $q_1$ have the opposite effect. The qualitative similarity between SSDG and ESFR schemes is once again apparent when considering the dissipation profiles of the former. Namely, increasing $c_{p-1}$ decreases dissipation at high wavenumbers and decreases the range over which the dissipation relation is accurate. In the limit $c_{p-1} \to \infty$, the dissipation vanishes over wavenumbers in the range $(p-1)\pi \leq \theta \leq (p+1)\pi$, as expected. Both EESFR and SSDG schemes are purely dissipative, \textit{ie.}, $\text{Im}(\omega^h) \leq 0$, which is consistent with the linear stability of both types of schemes.

All results presented for EESFR schemes are consistent with those reported by \cite{vermeire2016properties}. We refer the reader to this work for a more thorough discussion on the latter.
\begin{figure}
    \centering
    \begin{subfigure}{.49\textwidth}
        \centering
        \includegraphics[width=1\linewidth]{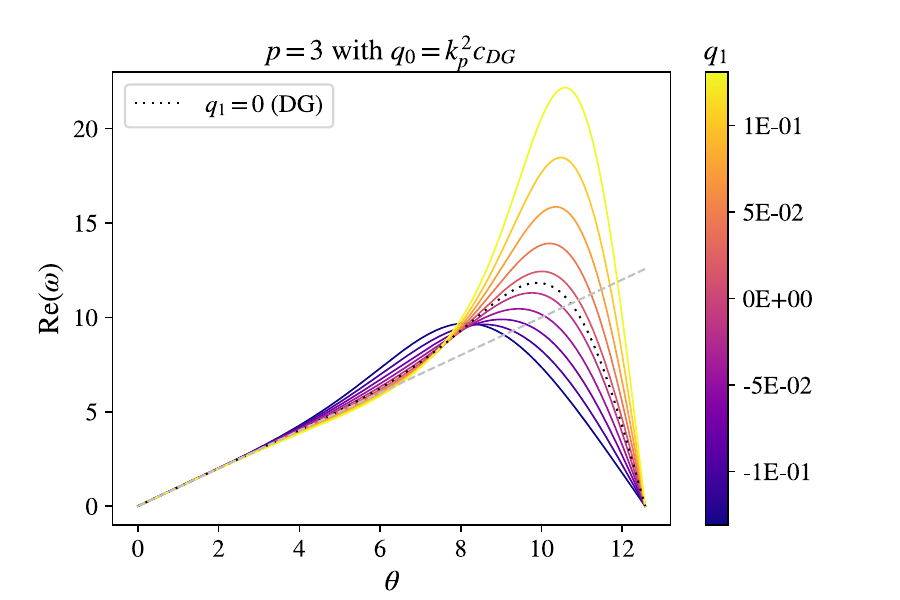}
        \caption{EESFR.}
    \end{subfigure}
    \begin{subfigure}{.49\textwidth}
        \centering
        \includegraphics[width=1\linewidth]{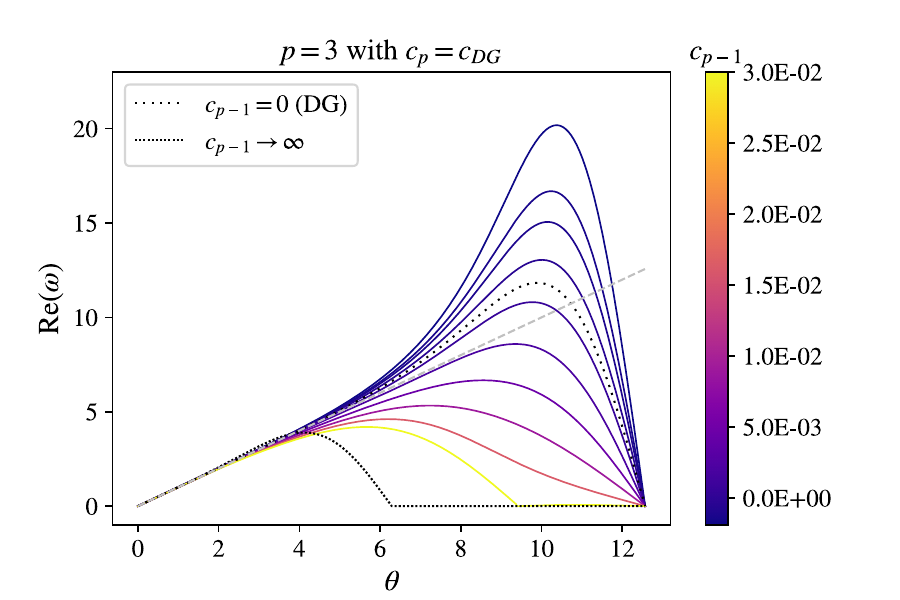}
        \caption{SSDG.}
    \end{subfigure}
    \caption{Dispersion relations for two-parameter EESFR and SSDG schemes with $p=3$, $c_{p} = c_{DG}$ and $q_0 = k_p^2 c_{DG}$ for an upwind numerical flux.}
    \label{fig:ch5_DG_dispers}
\end{figure}
\begin{figure}
    \centering
    \begin{subfigure}{.49\textwidth}
        \centering
        \includegraphics[width=1\linewidth]{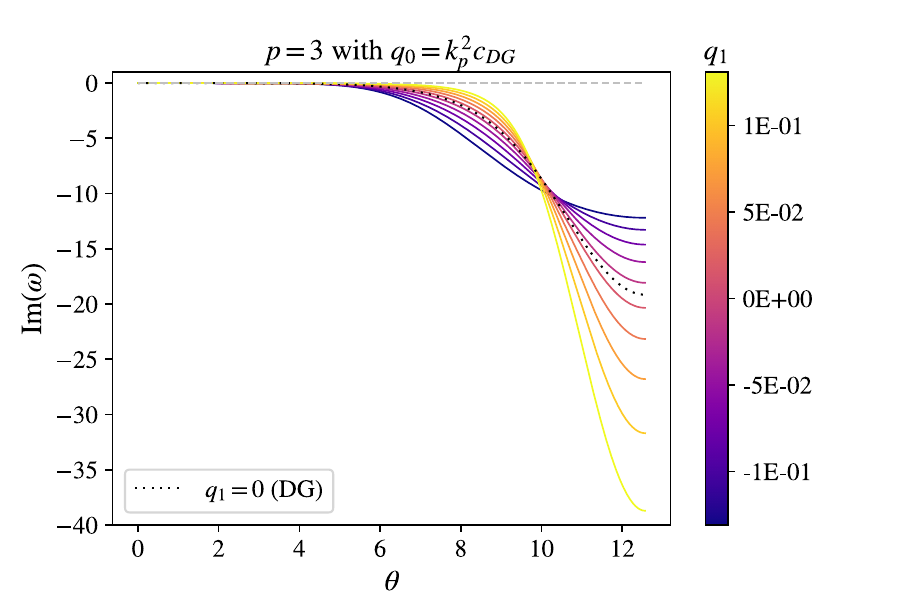}
        \caption{EESFR.}
    \end{subfigure}
    \begin{subfigure}{.49\textwidth}
        \centering
        \includegraphics[width=1\linewidth]{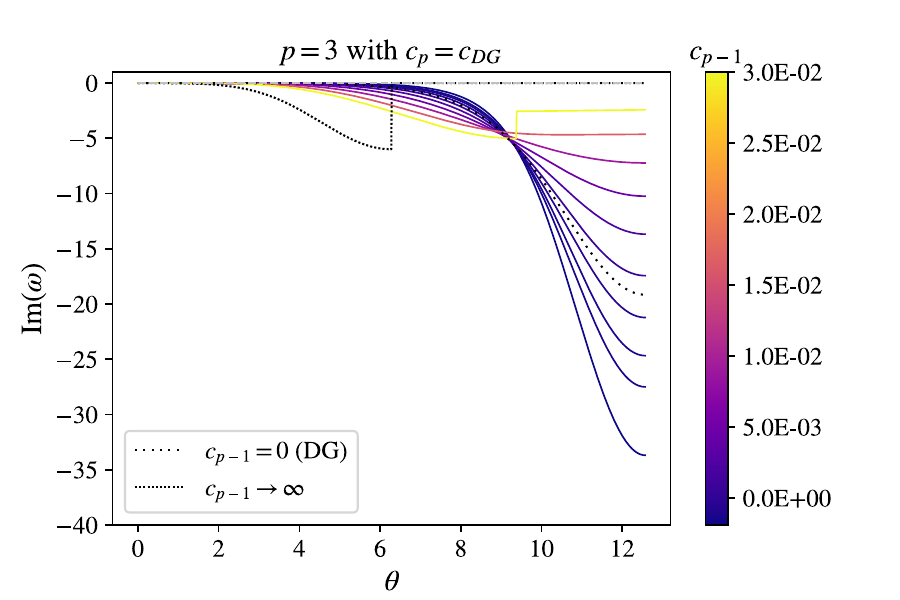}
        \caption{SSDG.}
    \end{subfigure}
    \caption{Dissipation relations for two-parameter EESFR and SSDG schemes with $p=3$, $c_{p} = c_{DG}$ and $q_0 = k_p^2 c_{DG}$ for an upwind numerical flux.}
    \label{fig:ch5_DG_dissip}
\end{figure}

\begin{figure}
    \centering
    \begin{subfigure}{.49\textwidth}
        \centering
        \includegraphics[width=1\linewidth]{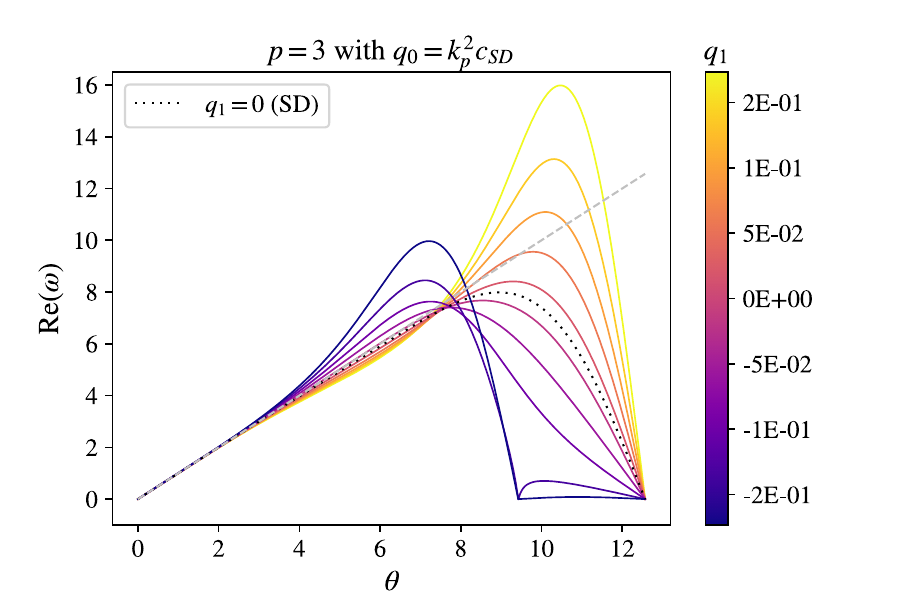}
        \caption{EESFR.}
    \end{subfigure}
    \begin{subfigure}{.49\textwidth}
        \centering
        \includegraphics[width=1\linewidth]{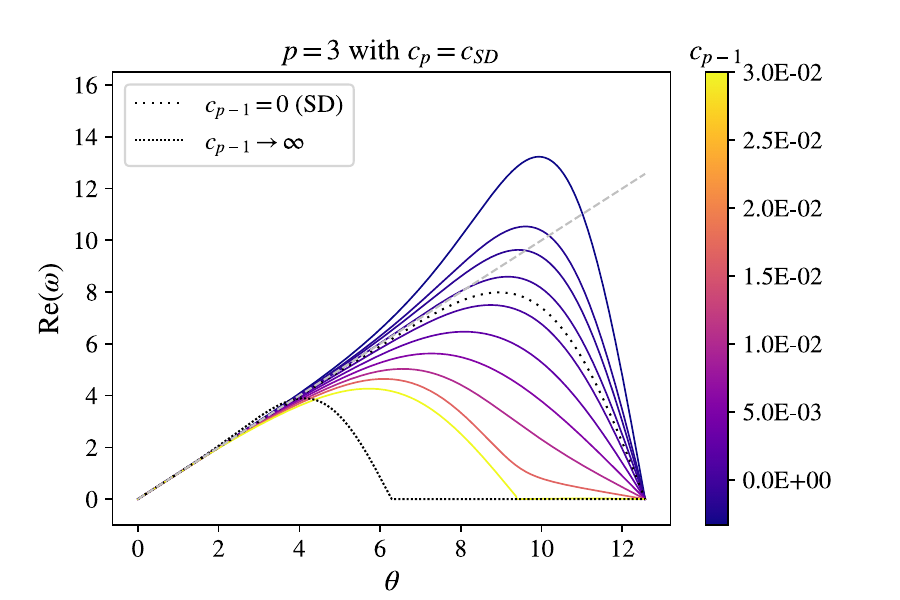}
        \caption{SSDG.}
    \end{subfigure}
    \caption{Dispersion relations for two-parameter EESFR and SSDG schemes with $p=3$, $c_{p} = c_{SD}$ and $q_0 = k_p^2 c_{SD}$ for an upwind numerical flux.}
    \label{fig:ch5_SD_dispers}
\end{figure}
\begin{figure}
    \centering
    \begin{subfigure}{.49\textwidth}
        \centering
        \includegraphics[width=1\linewidth]{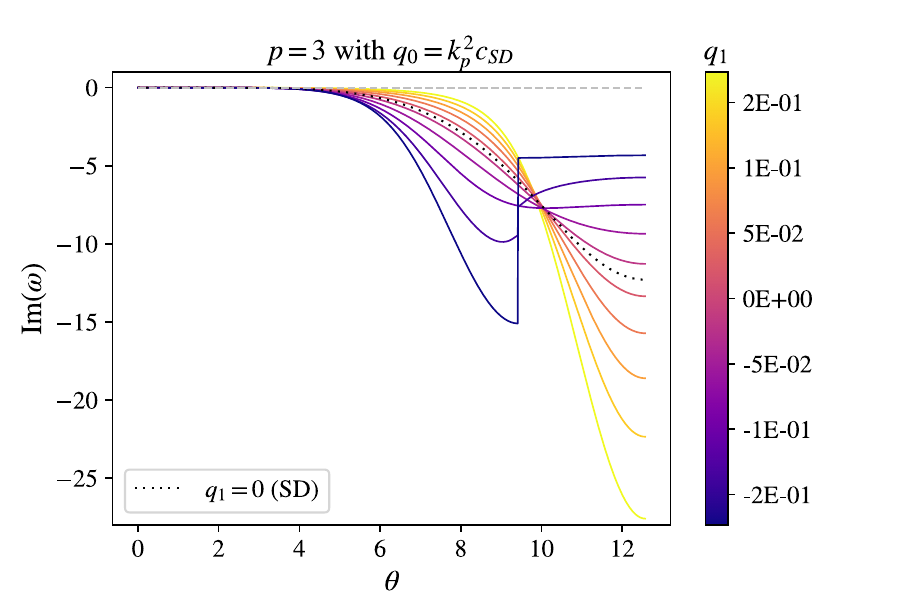}
        \caption{EESFR.}
    \end{subfigure}
    \begin{subfigure}{.49\textwidth}
        \centering
        \includegraphics[width=1\linewidth]{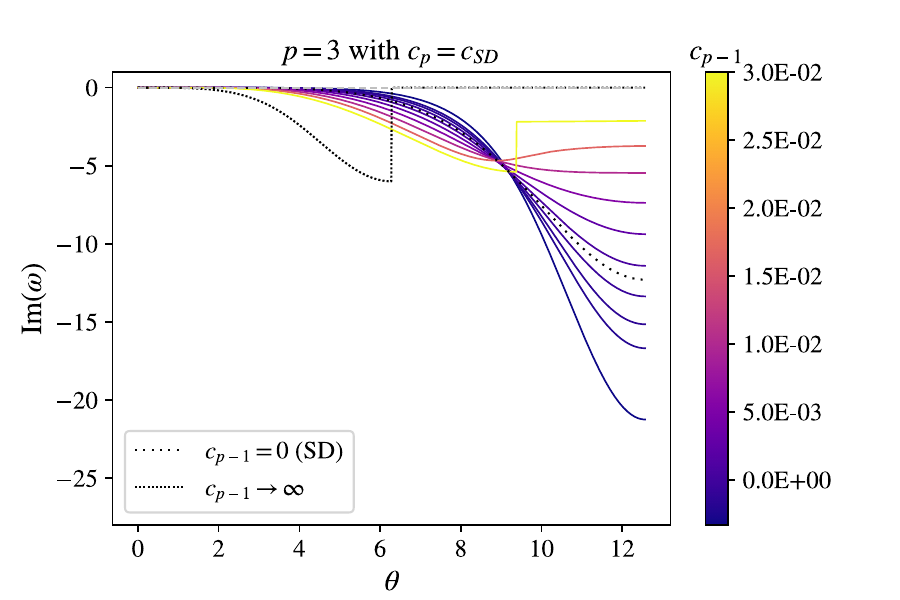}
        \caption{SSDG.}
    \end{subfigure}
    \caption{Dissipation relations for two-parameter EESFR and SSDG schemes with $p=3$, $c_{p} = c_{SD}$ and $q_0 = k_p^2 c_{SD}$ for an upwind numerical flux.}
    \label{fig:ch5_SD_dissip}
\end{figure}

\begin{figure}
    \centering
    \begin{subfigure}{.49\textwidth}
        \centering
        \includegraphics[width=1\linewidth]{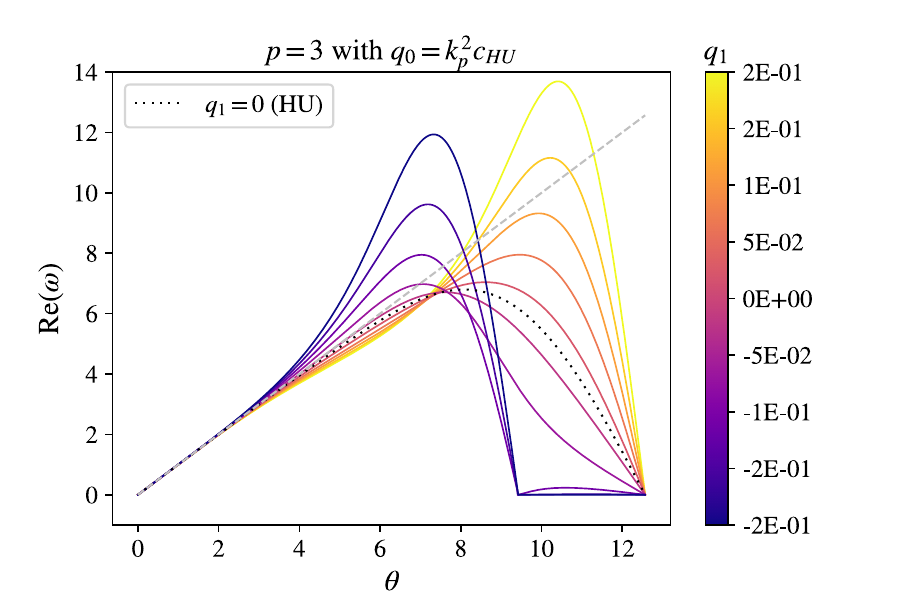}
        \caption{EESFR.}
    \end{subfigure}
    \begin{subfigure}{.49\textwidth}
        \centering
        \includegraphics[width=1\linewidth]{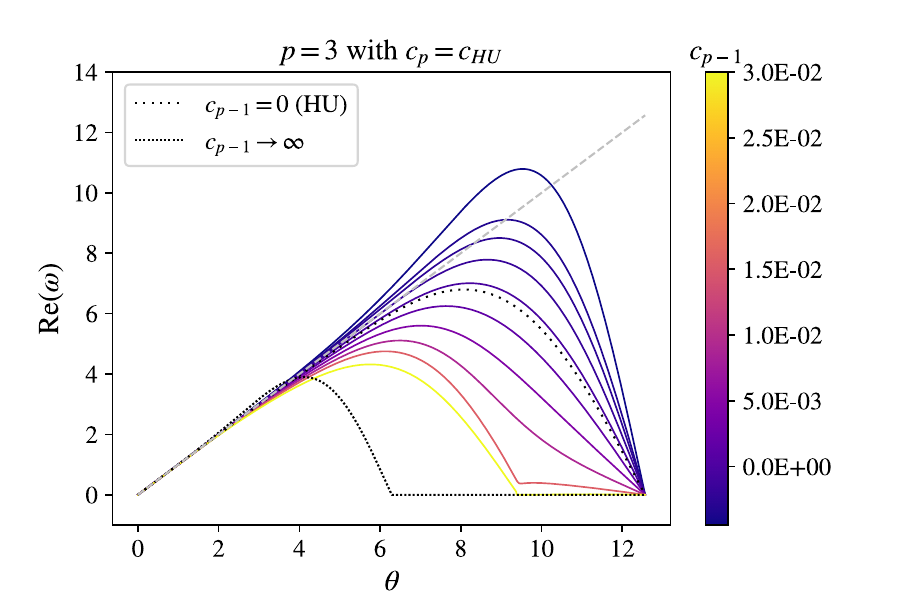}
        \caption{SSDG.}
    \end{subfigure}
    \caption{Dispersion relations for two-parameter EESFR and SSDG schemes with $p=3$, $c_{p} = c_{HU}$ and $q_0 = k_p^2 c_{HU}$ for an upwind numerical flux.}
    \label{fig:ch5_HU_dispers}
\end{figure}
\begin{figure}
    \centering
    \begin{subfigure}{.49\textwidth}
        \centering
        \includegraphics[width=1\linewidth]{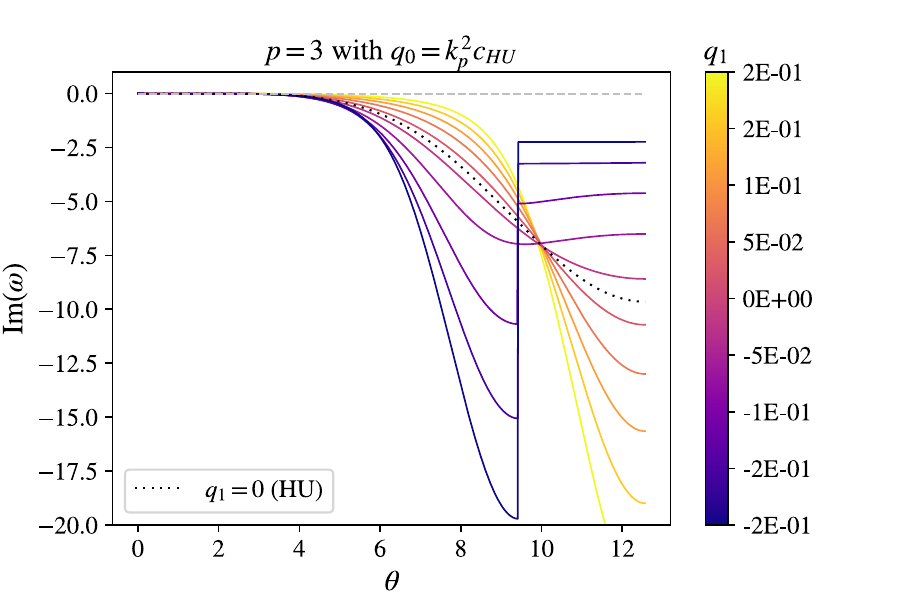}
        \caption{EESFR.}
    \end{subfigure}
    \begin{subfigure}{.49\textwidth}
        \centering
        \includegraphics[width=1\linewidth]{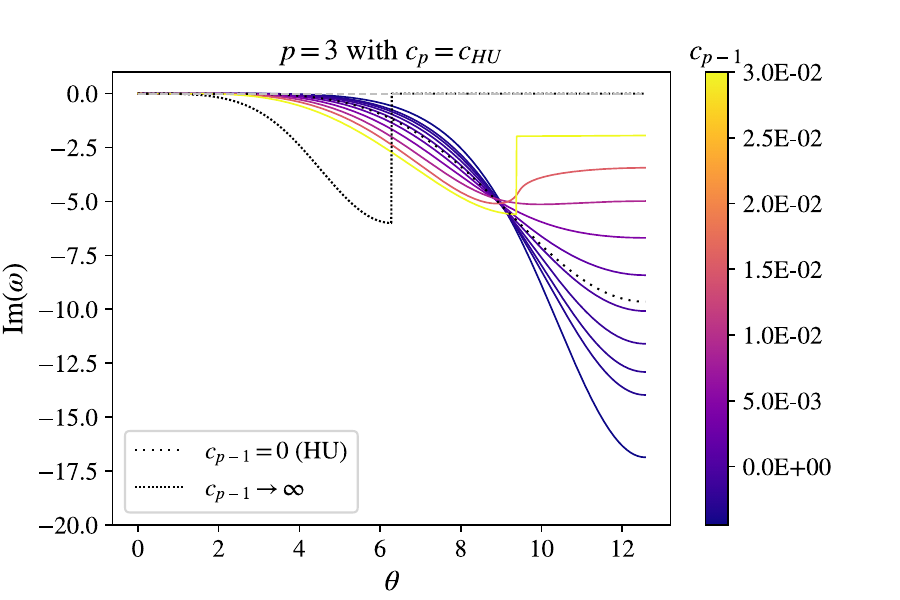}
        \caption{SSDG.}
    \end{subfigure}
    \caption{Dissipation relations for two-parameter EESFR and SSDG schemes with $p=3$, $c_{p} = c_{HU}$ and $q_0 = k_p^2 c_{HU}$ for an upwind numerical flux.}
    \label{fig:ch5_HU_dissip}
\end{figure}

\subsection{Spectral Order of Accuracy}
Figure \ref{fig:spectral_OOA} shows how the spectral order of accuracy of two-parameter EESFR and SSDG discretizations varies over a range of scheme parameters. Axes have been normalized in Figure \ref{fig:spectral_OOA_EESFR} to facilitate the comparison between the parameter spaces of EESFR and SSDG schemes. This normalization is based on Eq.(\ref{eq:param_scale}) and ensures that parameter values for EESFR and SSDG schemes coincide for $q_1 = c_{p-1} = 0$ and $q_0, q_1 \to \infty$. For clarity, the stability regions of both types of schemes have been shaded out. In this study, $\theta^h_R$ was set to $\pi / 4$, which is the value selected by \cite{huynh2007flux, vincent2011insights, vincent2015extended} for the analysis of third-order schemes. As can be seen, although both types of schemes feature fundamentally different dispersion-dissipation profiles and stability regions, their qualitative behavior appears to be nearly identical in the vicinity of $q_1 = 0$ and $c_{p-1} = 0$. When $q_1=0$ or $c_{p-1}=0$, the familiar features of ESFR schemes \cite{vincent2011insights} are recovered as expected. Namely, $A_T \approx 2p+1$ for the DG scheme and progressively decreases to $A_T \approx 2p-1$ as $q_0$ or $c_{p}$ are increased. However, although all schemes studied are superaccurate, ESFR-like behavior appears to be confined to a narrow ridge in the parameter space of EESFR and SSDG schemes. In both cases, any small deviation about $q_1 = 0$ or $c_{p-1} = 0$ leads the spectral order of accuracy to drop to $A_T \approx 2(p-1)$.

For SSDG schemes, it was observed that $A_T \approx 2(p-2)+1$ as $c_{p-1} \to \infty$, which corresponds to the spectral order of accuracy of a DG scheme of order $p-2$, as expected. A similar behavior does not exist for EESFR schemes, which are rendered unstable as $q_1 \to \infty$. For EESFR schemes, the results presented are consistent with those reported by \cite{vermeire2016properties}.
\begin{figure}
    \centering
    \begin{subfigure}{.49\textwidth}
        \centering
        \includegraphics[width=1\linewidth]{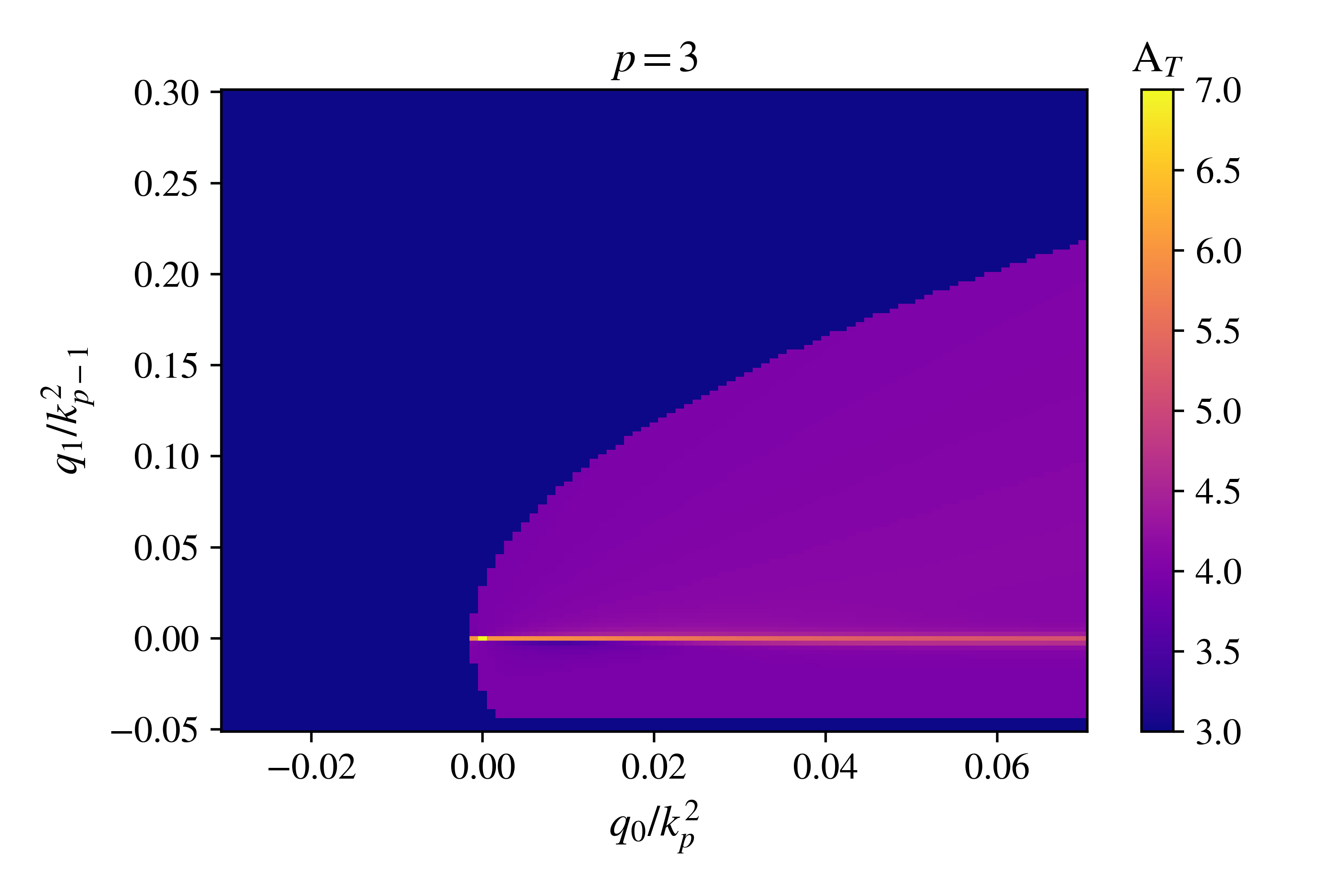}
        \caption{EESFR.}
        \label{fig:spectral_OOA_EESFR}
    \end{subfigure}
    \begin{subfigure}{.49\textwidth}
        \centering
        \includegraphics[width=1\linewidth]{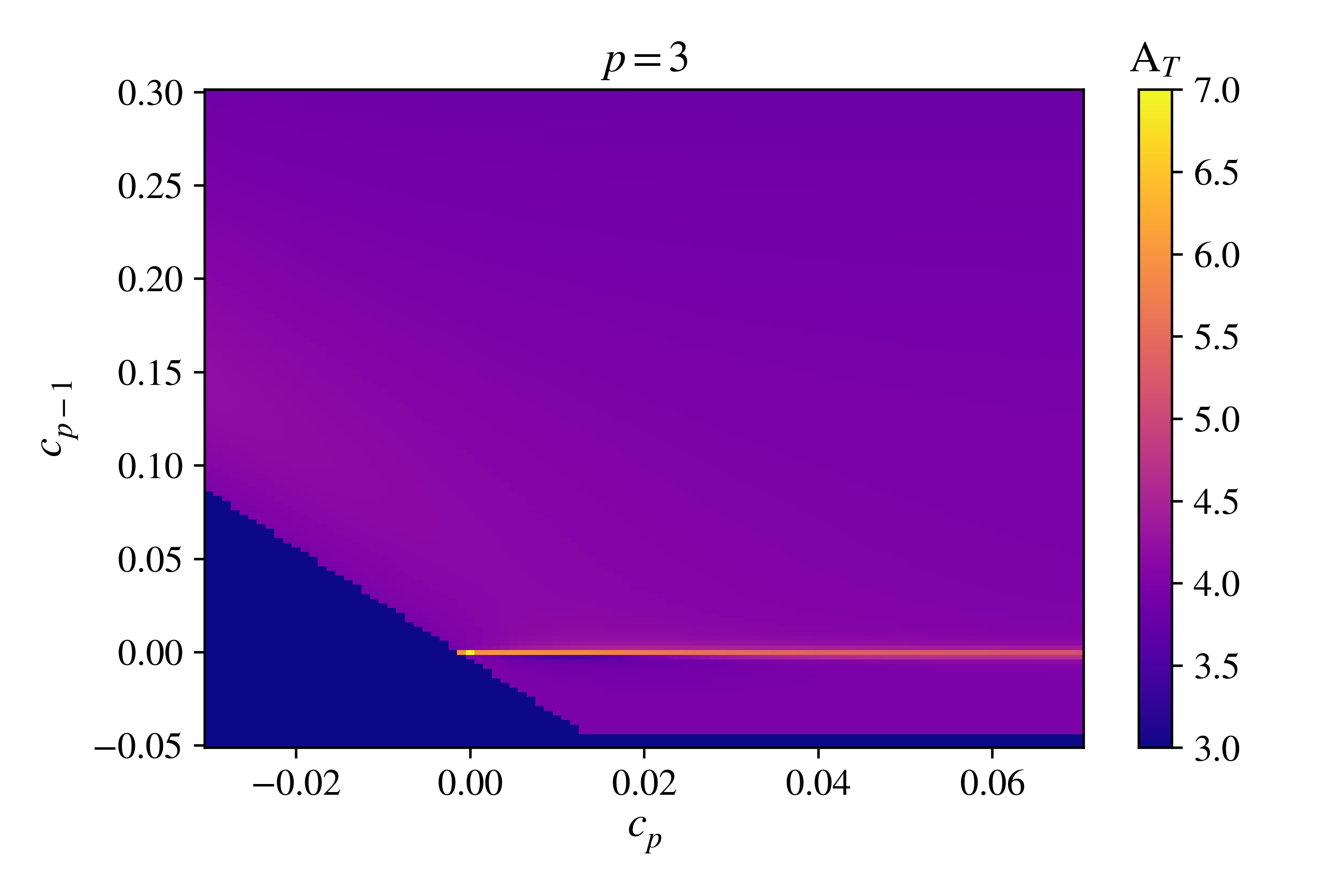}
        \caption{SSDG.}
    \end{subfigure}
    \caption{Spectral order of accuracy for two-parameter EESFR and SSDG schemes with $p=3$ and an upwind numerical flux.}
    \label{fig:spectral_OOA}
\end{figure}

\subsection{Stability}
Figures \ref{fig:CFL_RK33}, \ref{fig:CFL_RK44}, and \ref{fig:CFL_RK45} show the CFL time-step limit contours of two-parameter EESFR and SSDG schemes with respect to the scheme parameters for RK33, RK44 and RK45 explicit time-stepping schemes. As done previously, the axes are normalized for EESFR schemes to facilitate comparison with the parameter space of SSDG schemes. It is clear that the stability behavior of EESFR schemes differs significantly from that of SSDG schemes, although both frameworks can be employed to increase the CFL limit of ESFR schemes, no matter the type of time-stepping scheme used.

EESFR schemes feature a parabolic stability region with a narrow centered band of schemes associated with a high CFL limit. The CFL limit is maximized along that band for some positive and finite values of scheme parameters. As can be seen, the CFL limit varies non-monotonically with respect to both $q_0$ and $q_1$. It should be noted that the results presented in this paper regarding the locus of the maximum CFL for EESFR schemes differ from those reported by \cite{vermeire2016properties}. More precisely, \cite{vermeire2016properties} reports that the CFL limit for EESFR schemes is maximized in the limit $q_0 \to \infty$ when $q_1 = k_{p-1}^2c_{EESFR}^\tau$. It, however, appears that the more narrow range of scheme parameters considered by \cite{vermeire2016properties} did not allow them to observe the non-monotonicity of $\tau_{CFL}$ with respect to $q_1$ and thus prevented them from identifying the true global maximizers of the CFL contours for EESFR schemes. It should nevertheless be noted that when the range of scheme parameters considered is restricted to that studied by \cite{vermeire2016properties}, the stability contours presented in this paper agree with those reported by in their study.

For SSDG schemes, the stability region is polygonal and $\tau_{CFL}$ is maximized in a narrow band of high-CFL schemes located relatively close to the stability boundary described by Eq.(\ref{eq:ch5_stability_constraint2}). Schemes located in this region of high $\tau_{CFl}$ feature a negative value of $c_p$ and a relatively large positive value of $c_{p-1}$, which ensures their stability. For positive values of $c_p$ and $c_{p-1}$, SSDG schemes also feature a significant region of high CFL limit. Under the current normalization of scheme parameters, the extent of the latter significantly exceeds the stability region of EESFR schemes. As for EESFR schemes, the CFL limit of SSDG schemes varies non-monotonically with respect to both scheme parameters. Moreover, since both EESFR and SSDG schemes collapse to ESFR schemes of order $p-1$ as $q_0, c_p \to \infty,$ the behavior of $\tau_{CFL}$ for SSDG and EESFR schemes is identical in the limit $c_p, q_0 \to \infty$.

Maximum values of $\tau_{CFL}$ and their associated scheme parameters are recorded in Table \ref{tab:max_CFL} for EESFR and SSDG schemes of order 3 and 4. As can be seen, for all the temporal discretizations considered, the maximum CFL which is achievable via the EESFR and SSDG frameworks are nearly identical.
\begin{figure}
    \centering
    \begin{subfigure}{.49\textwidth}
        \centering
        \includegraphics[width=1\linewidth]{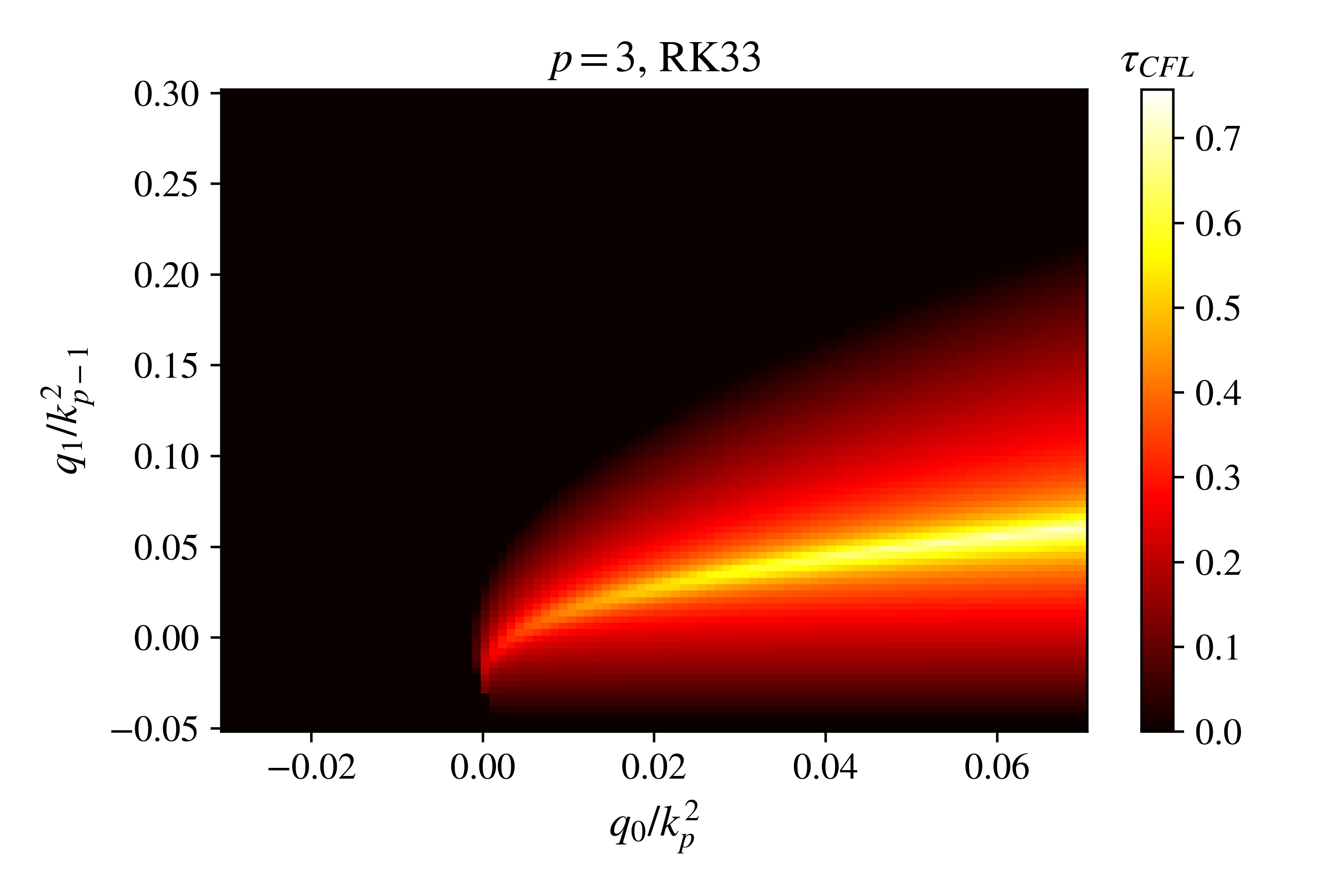}
        \caption{EESFR.}
    \end{subfigure}
    \begin{subfigure}{.49\textwidth}
        \centering
        \includegraphics[width=1\linewidth]{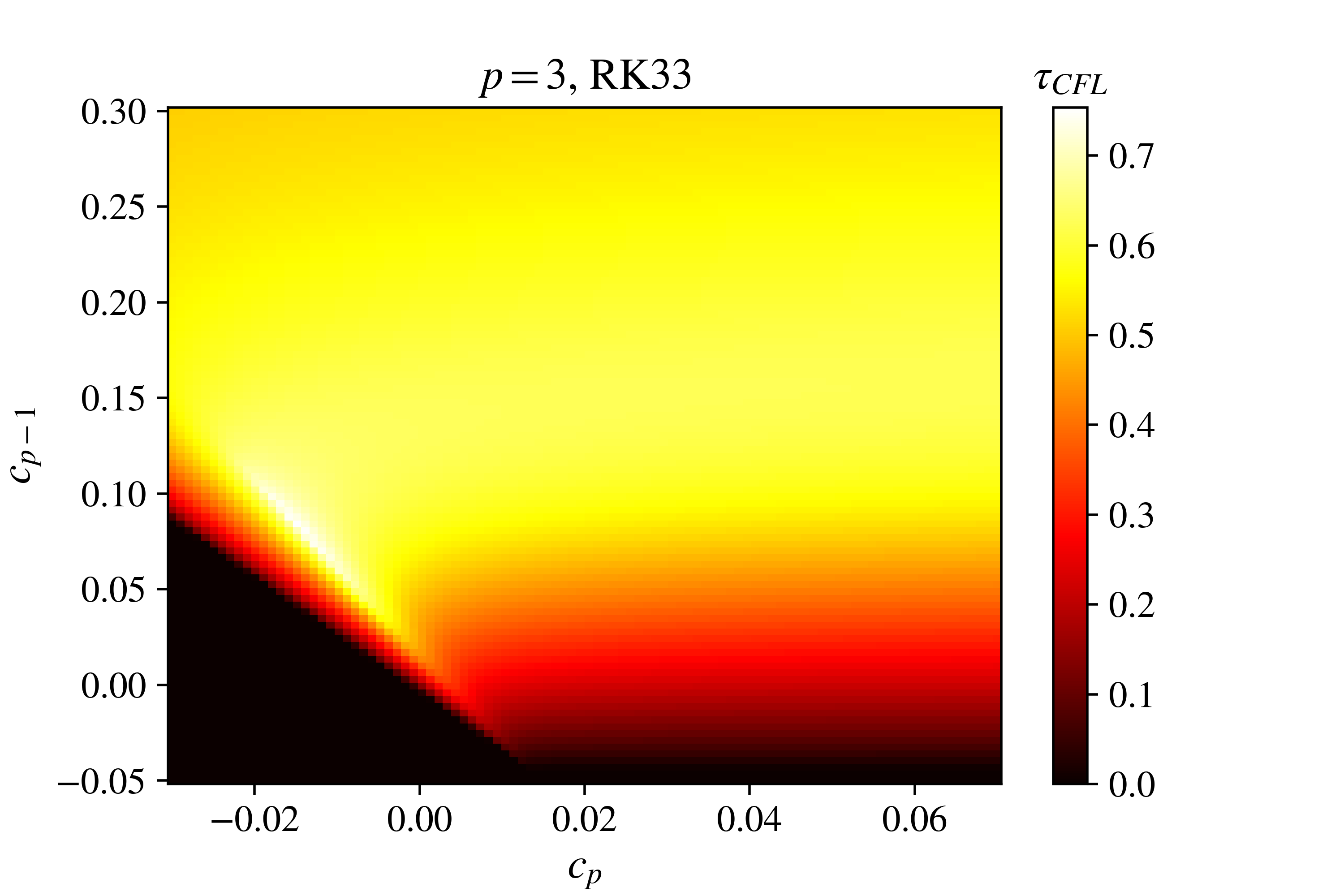}
        \caption{SSDG.}
    \end{subfigure}
    \caption{Explicit time-stepping limit for two-parameter EESFR and SSDG schemes with $p=3$, an upwind numerical flux and RK33 time-stepping.}
    \label{fig:CFL_RK33}
\end{figure}
\begin{figure}
    \centering
    \begin{subfigure}{.49\textwidth}
        \centering
        \includegraphics[width=1\linewidth]{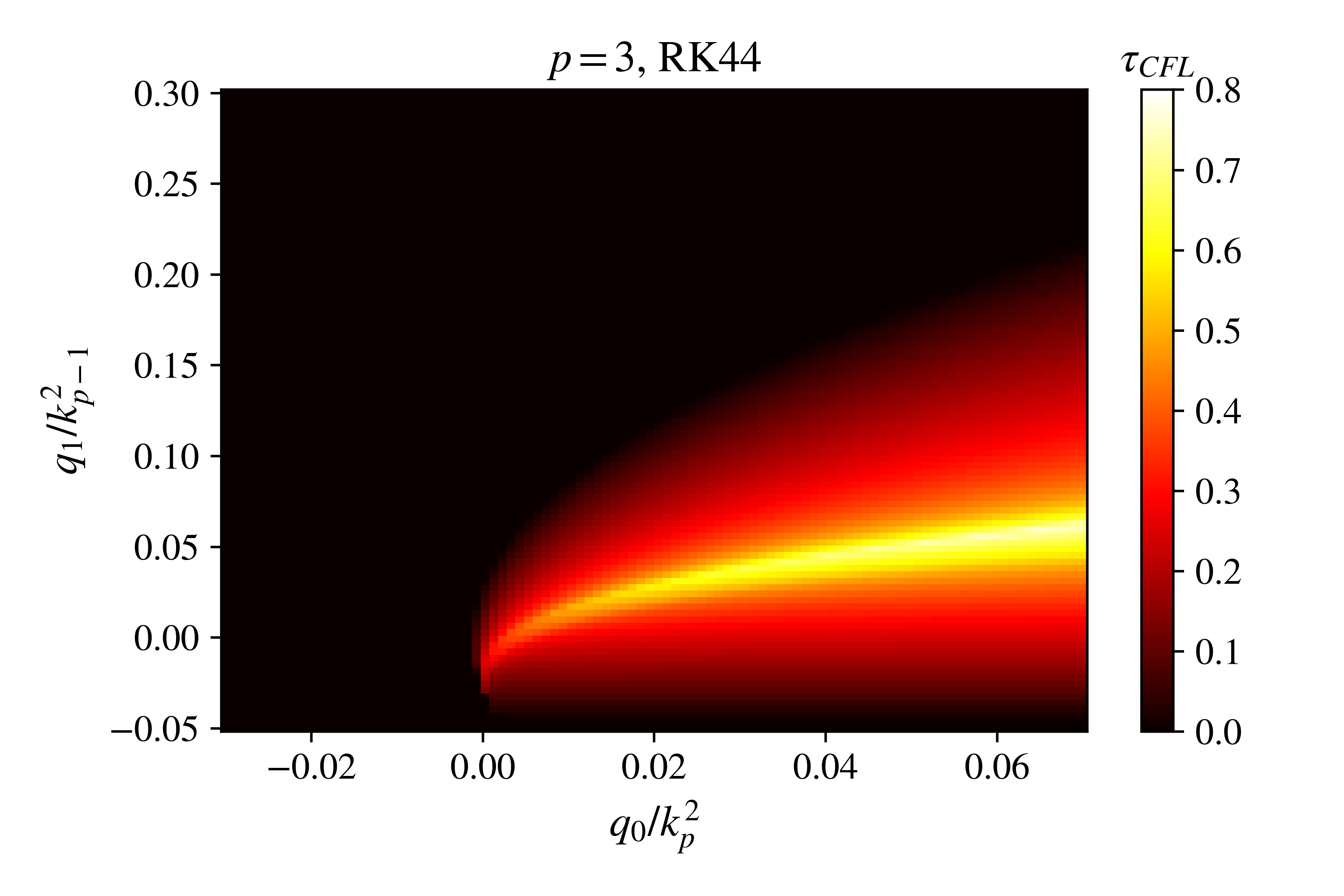}
        \caption{EESFR.}
    \end{subfigure}
    \begin{subfigure}{.49\textwidth}
        \centering
        \includegraphics[width=1\linewidth]{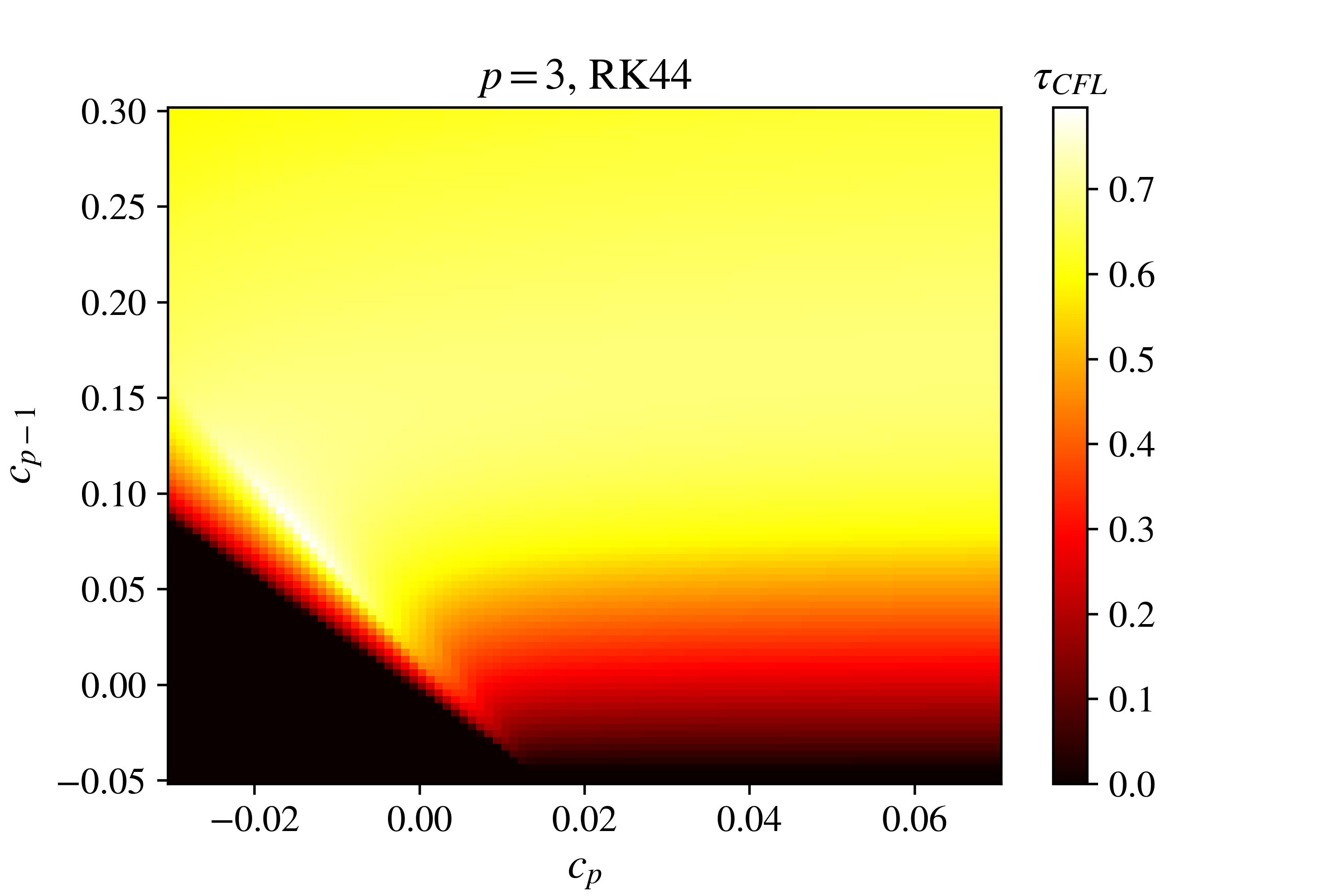}
        \caption{SSDG.}
    \end{subfigure}
    \caption{Explicit time-stepping limit for two-parameter EESFR and SSDG schemes with $p=3$, an upwind numerical flux and RK44 time-stepping.}
    \label{fig:CFL_RK44}
\end{figure}
\begin{figure}
    \centering
    \begin{subfigure}{.49\textwidth}
        \centering
        \includegraphics[width=1\linewidth]{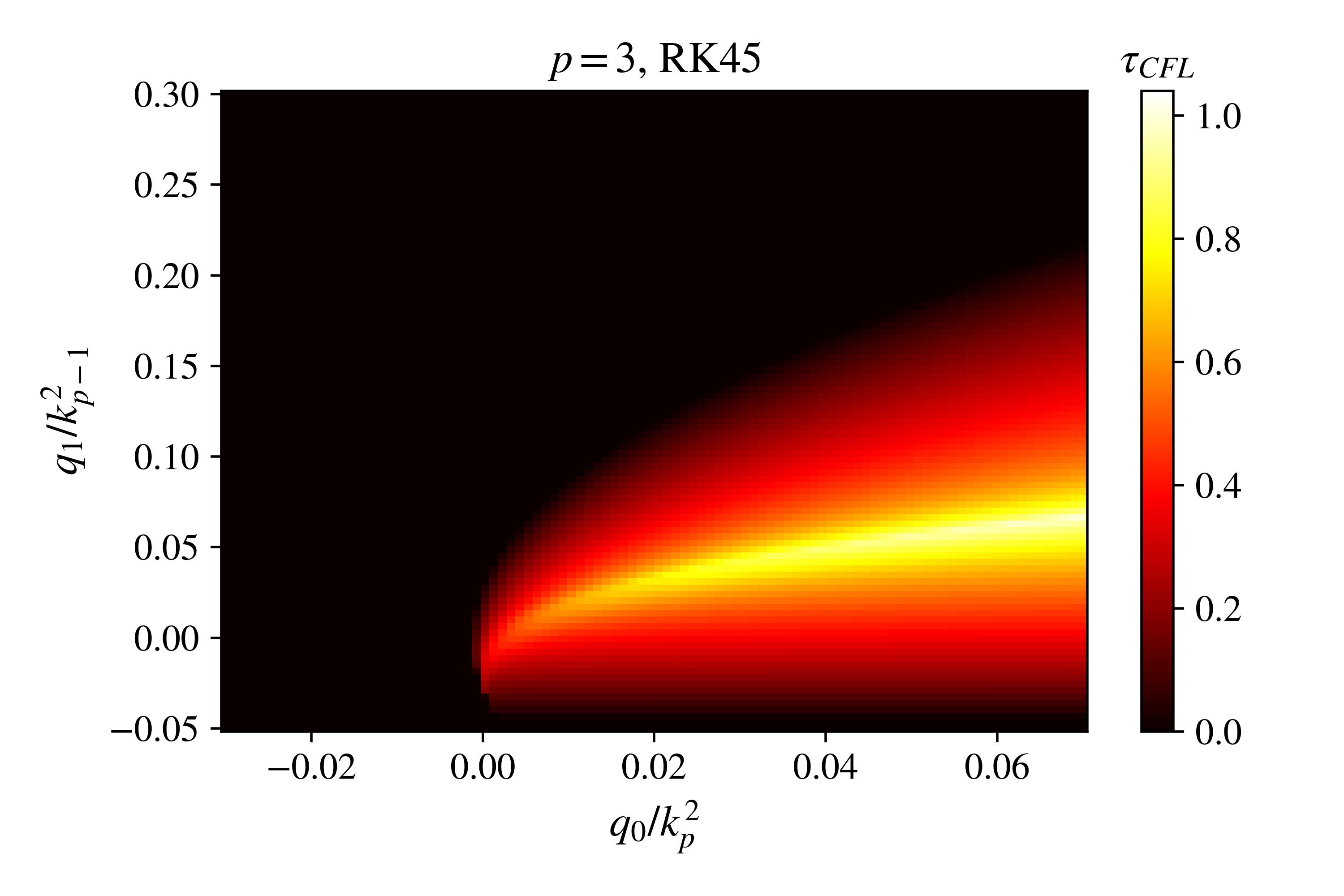}
        \caption{EESFR.}
    \end{subfigure}
    \begin{subfigure}{.49\textwidth}
        \centering
        \includegraphics[width=1\linewidth]{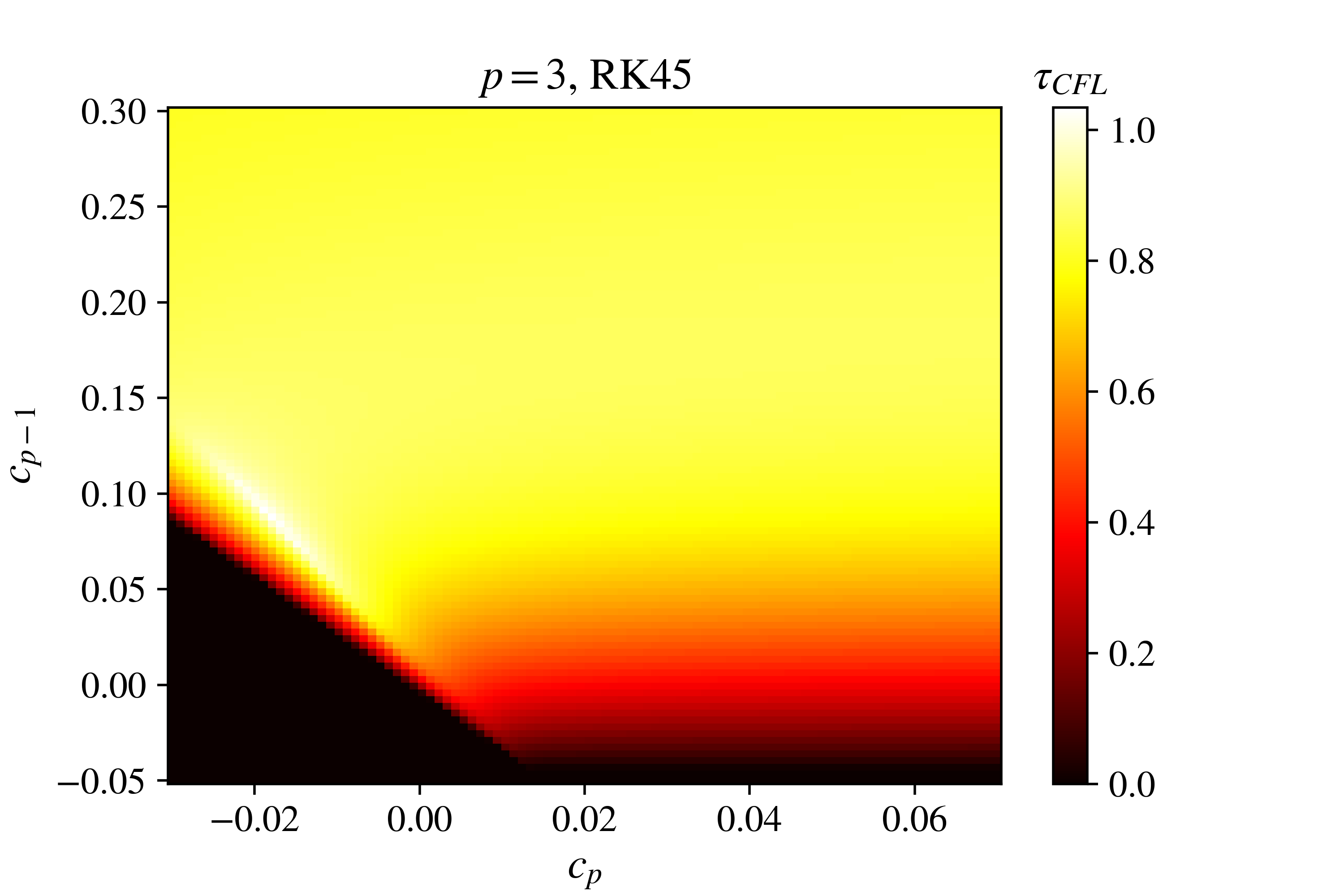}
        \caption{SSDG.}
    \end{subfigure}
    \caption{Explicit time-stepping limit for two-parameter EESFR and SSDG schemes with $p=3$, an upwind numerical flux and RK45 time-stepping.}
    \label{fig:CFL_RK45}
\end{figure}
\begin{table}
    \centering
    \begin{tabular}{cc|cc|cc}
         & & SSDG, $(c_{p}, c_{p-1})$  & $\tau_{CFL}$ & EESFR, $(q_0, q_1)$  & $\tau_{CFL}$  \\
         \hline
         $p=3$ & RK33  & ($-1.42 \times 10^{-2}$, $8.06 \times 10^{-2}$)  & 0.757 & ($29.4$, $0.761$)
         & 0.758 \\
         & RK44 & ($-1.52 \times 10^{-2}$, $8.36 \times 10^{-2}$) & 0.800 & ($29.6$, $0.772$) & 0.800 \\
         & RK45& ($-1.72 \times 10^{-2}$, $8.43 \times 10^{-2}$)  & 1.039 & ($24.9$, $0.757$) & 1.038 \\
         & &  & \\
         $p=4$ & RK33  & ($-3.70 \times 10^{-4}$, $1.54 \times 10^{-3}$)  &  0.413 & ($9.38$, $0.349$) & 0.413\\
         & RK44 & ($-3.76 \times 10^{-4}$, $1.56 \times 10^{-3}$) & 0.437 & ($9.23$, $0.350$) & 0.437  \\
         & RK45 & ($-4.00 \times 10^{-4}$, $1.57 \times 10^{-3}$)  & 0.565 & ($8.19$, $0.351$) & 0.565 \\
    \end{tabular}
    \caption{Maximum CFL time-step limit of EESFR and SSDG schemes of orders 3 and 4 for RK33, RK44 and RK45 explicit time-stepping schemes and an upwind numerical flux.}
    \label{tab:max_CFL}
\end{table}

\subsection{Comparison of Schemes}
To complete the characterization and comparison of the linear properties of EESFR and SSDG schemes, we consider plotting the CFL time-step limit versus the order of accuracy of the schemes for key values of $c_p$ and $q_0$. In this case, we compute the order of accuracy by solving the linear advection problem with $a=2$ on the periodic domain $[-\pi, \pi]$ for initial conditions given by
\begin{equation}
    u(x,0) = \sin(x).
\end{equation}
The solution is evolved until $t = \pi$ and the L2 error is computed using a quadrature of sufficient strength while the mesh is progressively refined. The maximum number of elements used for these purposes was $N=256$. It should be noted that this approach differs from that previously discussed, as the order of accuracy computed in this section includes all types of errors pertaining to the numerical scheme, whereas $A_T$ only accounts for errors arising from dispersion and dissipation.

Results for EESFR and SSDG schemes for $c_p \in \{c_{DG}, c_{SD}, c_{HU}\}$ and $q_0 \in \{k_p^2c_{DG}, k_p^2c_{SD}, k_p^2c_{HU}\}$ are shown in Figures \ref{fig:CFL_OOA_DG}, \ref{fig:CFL_OOA_SD} and \ref{fig:CFL_OOA_HU}. For concision, only data for RK44 explicit time-stepping is included. Once again, it is clear that the behavior of EESFR schemes differs greatly from that of SSDG schemes. While the EESFR schemes considered are able to retain an order of accuracy of $p+1$ for $q_1 \neq 0$, the order of accuracy of SSDG schemes drops rapidly from $p+1$ to $p$ whenever $c_{p-1} \neq 0$. This result is somewhat to be expected as FR schemes are effectively constructed by filtering the face-terms of the classical DG scheme (\textit{ie.}, by introducing a modified lifting operator), while FDG schemes are obtained by applying a linear filter to the complete DG residual. Moreover, for the values of $c_p$ and $q_0$ considered, it is clear that SSDG schemes are able to reach much higher values of $\tau_{CFL}$, which is consistent with the high stability of these schemes, which can be observed over a large region in the associated parameter space. In fact, in all cases, the maximum value of $\tau_{CFL}$ that can be achieved with ESFR schemes can be exceeded with SSDG schemes, which is not the case with EESFR schemes, as shown in Figures \ref{fig:CFL_OOA_DG}, \ref{fig:CFL_OOA_SD} and \ref{fig:CFL_OOA_HU}.
\begin{figure}
    \centering
    \begin{subfigure}{.49\textwidth}
        \centering
        \includegraphics[width=1\linewidth]{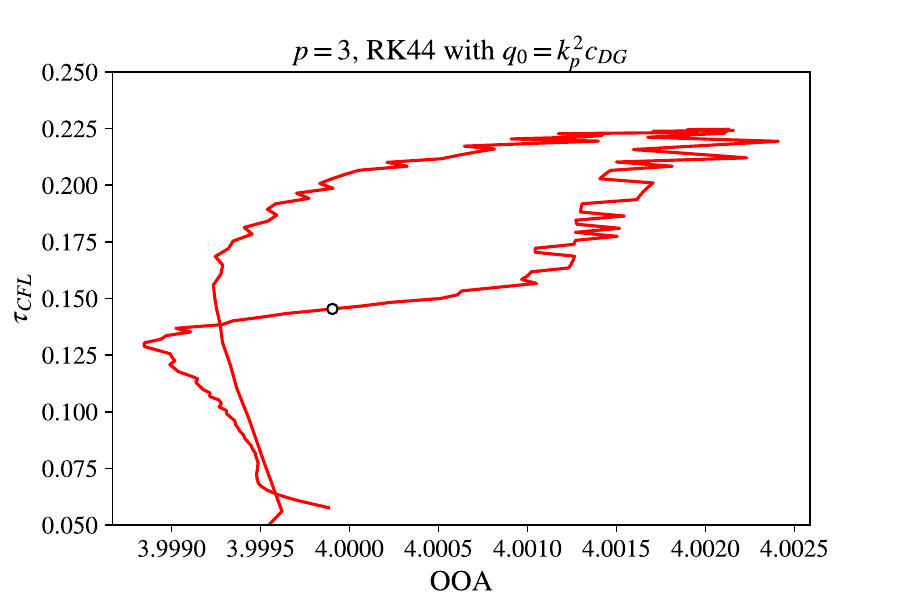}
        \caption{EESFR.}
    \end{subfigure}
    \begin{subfigure}{.49\textwidth}
        \centering
        \includegraphics[width=1\linewidth]{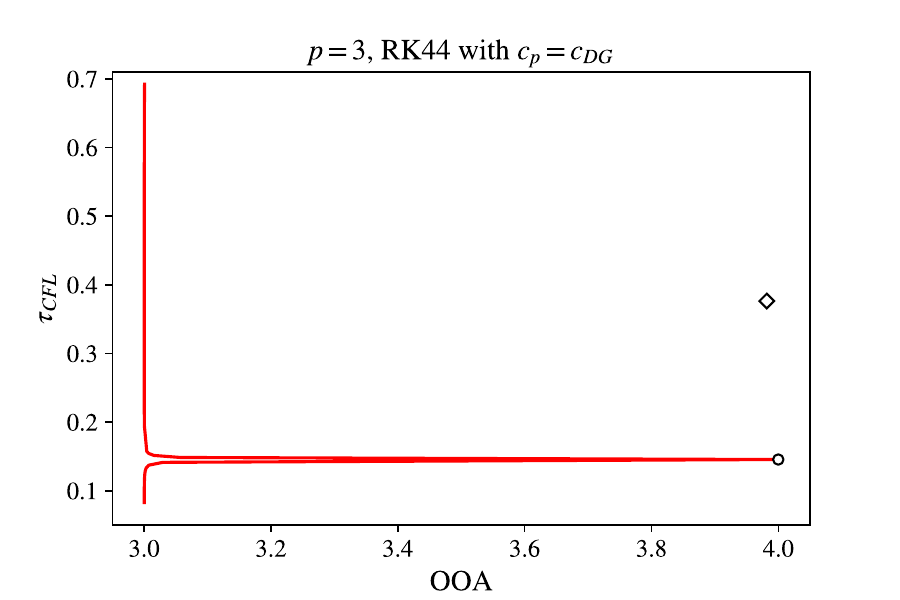}
        \caption{SSDG.}
    \end{subfigure}
    \caption{CFL time-stepping limit versus order of accuracy for EESFR and SSDG schemes with $c_p = c_{DG}$ and $q_0 = k_p^2c_{DG}$, $p=3$, an upwind numerical flux and RK44 time-stepping. Circle is at $c_{p-1}, q_1 = 0$ (DG), and diamond denotes the ESFR scheme with the maximum CFL limit.}
    \label{fig:CFL_OOA_DG}
\end{figure}
\begin{figure}
    \centering
    \begin{subfigure}{.49\textwidth}
        \centering
        \includegraphics[width=1\linewidth]{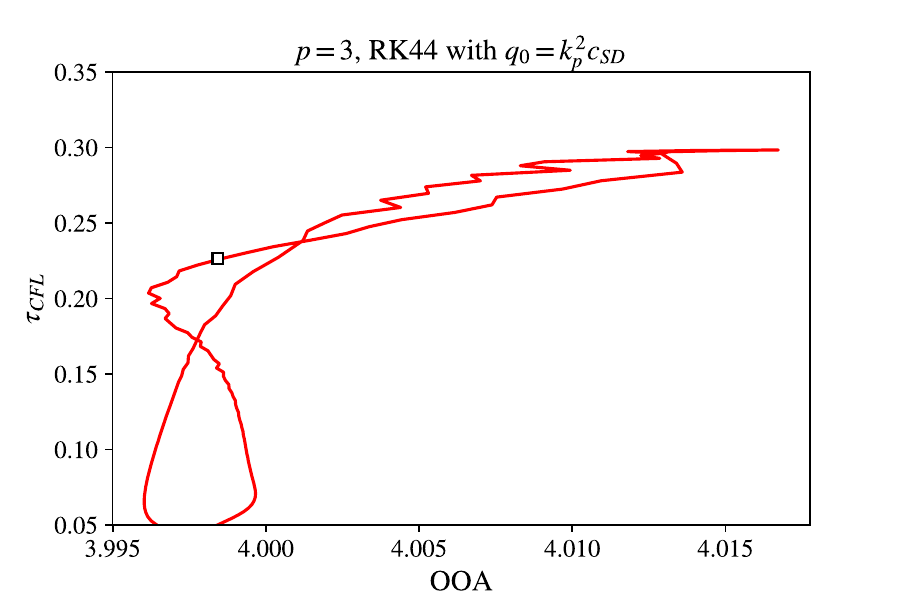}
        \caption{EESFR.}
    \end{subfigure}
    \begin{subfigure}{.49\textwidth}
        \centering
        \includegraphics[width=1\linewidth]{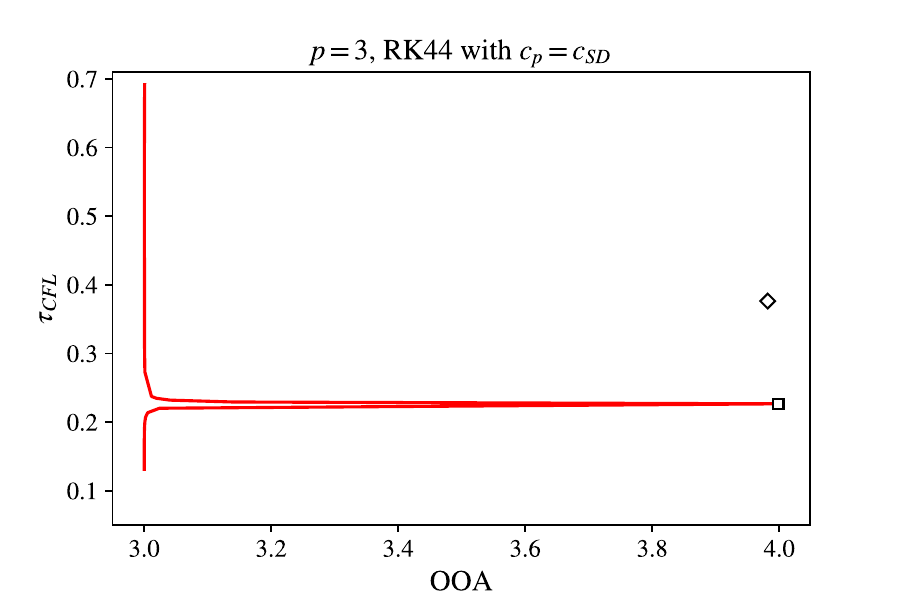}
        \caption{SSDG.}
    \end{subfigure}
    \caption{CFL time-stepping limit versus order of accuracy for EESFR and SSDG schemes with $c_p = c_{SD}$ and $q_0 = k_p^2c_{SD}$, $p=3$, an upwind numerical flux and RK44 time-stepping. Square is at $c_{p-1}, q_1 = 0$ (SD), and diamond denotes the ESFR scheme with the maximum CFL limit.}
    \label{fig:CFL_OOA_SD}
\end{figure}
\begin{figure}
    \centering
    \begin{subfigure}{.49\textwidth}
        \centering
        \includegraphics[width=1\linewidth]{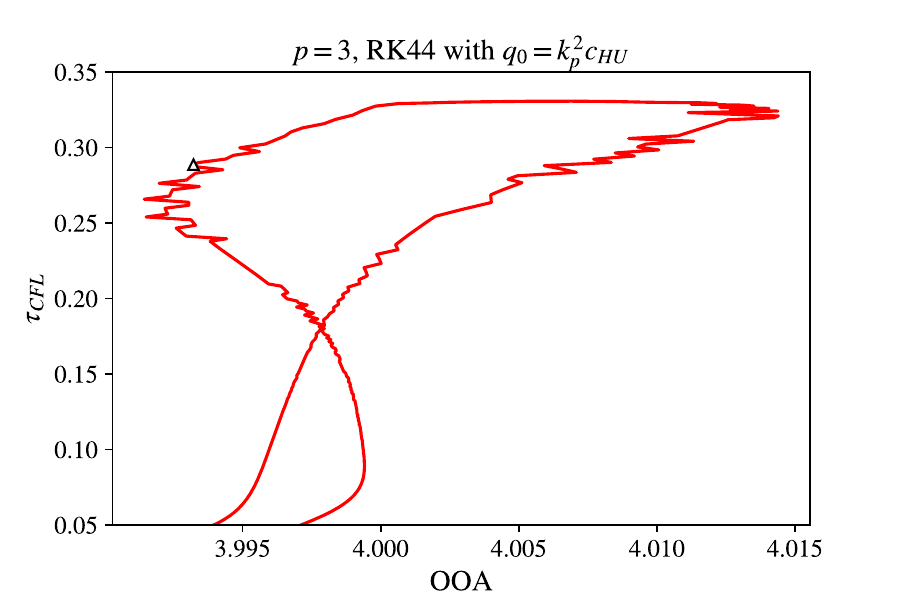}
        \caption{EESFR.}
    \end{subfigure}
    \begin{subfigure}{.49\textwidth}
        \centering
        \includegraphics[width=1\linewidth]{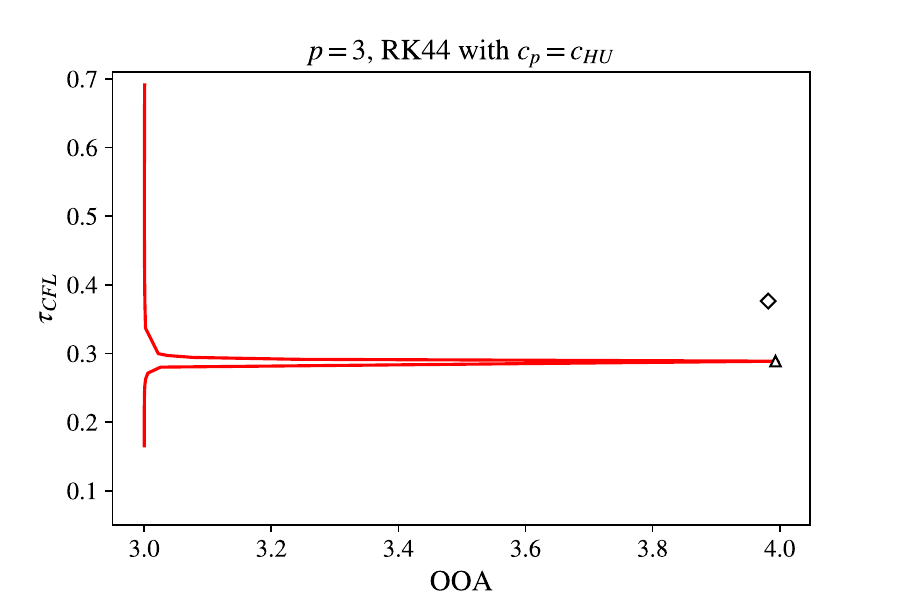}
        \caption{SSDG.}
    \end{subfigure}
    \caption{CFL time-stepping limit versus order of accuracy for EESFR and SSDG schemes with $c_p = c_{HU}$ and $q_0 = k_p^2c_{HU}$, $p=3$, an upwind numerical flux and RK44 time-stepping. Triangle is at $c_{p-1}, q_1 = 0$ (HU), and diamond denotes the ESFR scheme with the maximum CFL limit.}
    \label{fig:CFL_OOA_HU}
\end{figure}

Finally, using the linear advection test case previously described, the L2 error for the third-order EESFR and SSDG schemes yielding the maximum RK44 CFL time-step limit (see Table \ref{tab:max_CFL}) was also computed as the mesh was progressively refined. Results are shown in Figure \ref{fig:error_CFLmax}. As can be seen, while both schemes are associated with the same CFL limit of $\tau_{CFL} = 0.800$, the EESFR scheme achieves an order of accuracy of $p+1$ while the SSDG scheme is only $p$th order accurate.

Given that the order and the CFL limit significantly affect the computational cost of a given numerical scheme, it is tempting to use the results from the linear study presented in this section to infer that EESFR schemes are more suitable than SSDG schemes for CFD applications. It is, however, the opinion of the authors of this work that this former statement should be nuanced in the light of the objectives of this paper. Namely, by design, SSDG schemes are particularly simple to implement in existing DG codes and extension of the simple 1D formulation presented in this paper to triangular elements is straightforward following the work of \cite{zwanenburg2016equivalence}. Moreover, the FDG structure of the SSDG framework ensures compatibility of the latter with the NSFR approach developed by \cite{cicchino2022nonlinearly}, which is not the case for EESFR schemes. Finally, the comparatively large range of scheme parameters over which the CFL limit can be increased through the SSDG framework could be an interesting feature for applications in which the CFL limit is a significant bottleneck.
\begin{figure}
        \centering
        \includegraphics[width=0.7\linewidth]{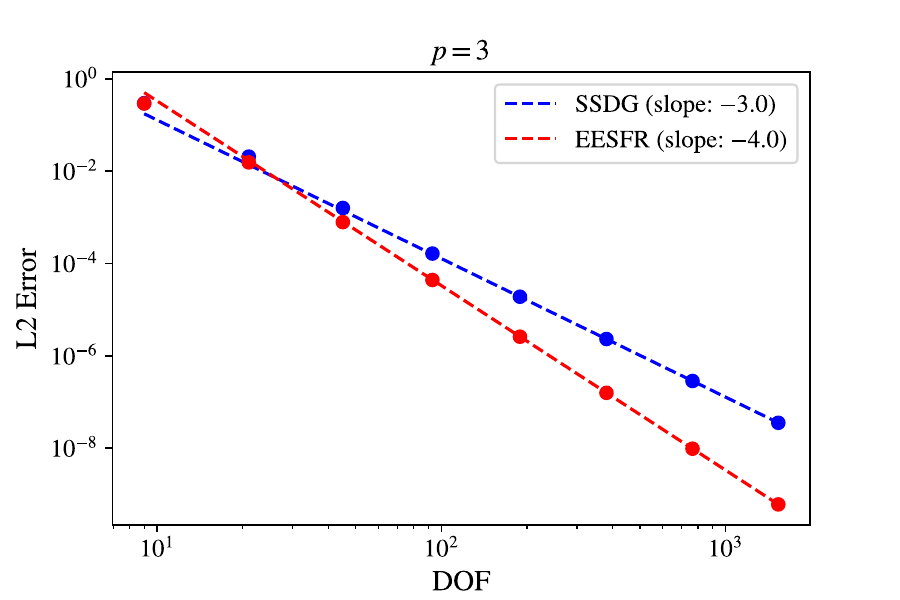}
        \caption{L2 Error for $p=3$ EESFR and SSDG schemes yielding the largest RK44 CFL time-step limit.}
        \label{fig:error_CFLmax}
\end{figure}

\section{Conclusion}
In this work, using the FDG framework, we have introduced SSDG schemes, a new conservative and linearly stable generalization of ESFR schemes wich is compatible with the NSFR approach developed by \cite{cicchino2022nonlinearly}. For general numerical fluxes, the schemes generated via the EESFR and SSDG frameworks are distinct and only overlap for classical ESFR schemes.

The linear properties of two-parameter EESFR and SSDG schemes were moreover studied. Using Von Neumann analysis, it was found that while SSDG and EESFR schemes exhibit fundamentally different dispersive and dissipative behaviors, they can achieve very similar increases in CFL limit for all the types of explicit time-stepping schemes considered. The range of scheme parameters over which the CFL limit could be improved was, however, much larger for SSDG schemes than for EESFR schemes. Additionally, when considering the spectral order of accuracy, EESFR and SSDG schemes were observed to behave in a qualitatively identical way in the vicinity of $c_{p-1} = q_1 = 0$.

Finally, for the 1D linear advection problem, it was found that SSDG schemes were only $p$th order accurate under $h$-refinement when $c_{p-1} \neq 0$ while EESFR schemes could achieve $(p+1)$th order accuracy for $q_1 \neq 0$.

Future studies should include the relatively trivial extension of SSDG schemes to the NSFR framework developed by \cite{cicchino2022nonlinearly}. The feasibility of extending EESFR schemes to triangular and tetrahedral elements and to the NSFR framework should moreover be examined. Applications of SSDG schemes for shock capturing purposes, as investigated by \cite{shruthi2025investigation} for ESFR schemes, could also be considered.

\appendix
\section{Connection between FR and FDG Schemes}
\label{app1}
\begin{lemma}
\label{lem:equivalence}
    A FDG scheme can be expressed in FR form if and only if its associated filtering matrix $\mathbf{K} \in \mathbb{R}^{(p+1) \times (p+1)}$ is such that $\mathbf{M} + \mathbf{K}$ is invertible and
    \begin{equation}
        \mathbf{K}\mathbf{D} = 0.
    \end{equation}
    Additionally, a FR scheme can be expressed in FDG form if and only if its associated matrix $\mathbf{Q} \in \mathbb{R}^{(p+1) \times (p+1)}$ is such that $\mathbf{M} + \mathbf{Q}$ is invertible and $\exists \mathbf{K} \in \mathbb{R}^{(p+1) \times (p+1)}$ with $\mathbf{M} + \mathbf{K}$ invertible satisfying
    \begin{align}
    \mathbf{K}\mathbf{D} &= 0, \\
    (\mathbf{M} + \mathbf{K})^{-1} \mathbf{r} &= (\mathbf{M} + \mathbf{Q})^{-1} \mathbf{r},\\
    (\mathbf{M} + \mathbf{K})^{-1} \mathbf{l} &= (\mathbf{M} + \mathbf{Q})^{-1} \mathbf{l}.
\end{align}
\end{lemma}
\begin{proof}
Suppose that a scheme can be expressed both in FR and FDG form. Then, the two representations for this numerical scheme must admit the same Bloch wave solutions. Consequently, one has
\begin{equation}
    \forall \theta \in [-(p+1)\pi, (p+1)\pi] : \mathbf{H}_{FR}(\theta) = \mathbf{H}_{FDG}(\theta).
\end{equation}
Using Eq.(\ref{eq:ch3_Q_FDG}) and Eq.(\ref{eq:ch3_Q_FR}), this implies
\begin{align}
    (\mathbf{M}+\mathbf{K})^{-1}\mathbf{M}\mathbf{D}
    - \frac{1}{2}\alpha (\mathbf{M}+\mathbf{K})^{-1}\mathbf{r}\mathbf{r}^T
    + \frac{1}{2}(2-\alpha) (\mathbf{M}+\mathbf{K})^{-1} \mathbf{l}\mathbf{l}^T&
    \\ \nonumber
    - \mathbf{D}
    + \frac{1}{2}\alpha(\mathbf{M}+\mathbf{Q})^{-1} \mathbf{r}\mathbf{r}^T
    - \frac{1}{2}(2-\alpha)(\mathbf{M}+\mathbf{Q})^{-1} \mathbf{l}\mathbf{l}^T&
    \nonumber \\
    +
    \left(
    \frac{1}{2}\alpha (\mathbf{M}+\mathbf{K})^{-1} \mathbf{r} \mathbf{l}^T
    -\frac{1}{2}\alpha (\mathbf{M} + \mathbf{Q})^{-1} \mathbf{r} \mathbf{l}^T
    \right) e^{I\theta^h}& \nonumber \\
    +
    \left(
    \frac{1}{2}(2-\alpha) (\mathbf{M} + \mathbf{Q})^{-1} \mathbf{l} \mathbf{r}^{T}
    -\frac{1}{2}(2-\alpha) (\mathbf{M}+\mathbf{K})^{-1} \mathbf{l} \mathbf{r}^T
    \right) e^{-I\theta^h}&
    =
    0.
\end{align}
By linear independence of $\{1,e^{I\theta}, e^{-I\theta}\}$, all terms in parentheses must vanish. Hence, if $\alpha \neq 0$, we conclude that
\begin{align}
    \mathbf{K}\mathbf{D} &= 0 \label{eq:equivalence1}, \\
    (\mathbf{M} + \mathbf{K})^{-1} \mathbf{r} &= (\mathbf{M} + \mathbf{Q})^{-1} \mathbf{r}, \label{eq:equivalence2}\\
    (\mathbf{M} + \mathbf{K})^{-1} \mathbf{l} &= (\mathbf{M} + \mathbf{Q})^{-1} \mathbf{l} \label{eq:equivalence3}.
\end{align}
Eq.(\ref{eq:equivalence1}), Eq.(\ref{eq:equivalence2}) and Eq.(\ref{eq:equivalence3}) form a set of necessary conditions for equivalence between FR and FDG schemes. Sufficiency of the latter can easily be verified by direct substitution in the definition of the schemes.
\end{proof}

\begin{theorem}
    If a given numerical scheme that can be expressed in FR and FDG form is such that $\mathbf{K} = \mathbf{K}^T$ or $\mathbf{J}\mathbf{Q} = \mathbf{Q}\mathbf{J}$, then it must be an ESFR scheme, \textit{ie.}, ESFR schemes are the only numerical schemes satisfying the conditions from Lemma \ref{lem:equivalence} with symmetric filtering operators or symmetric correction functions.
\end{theorem}
\begin{proof}
Without loss of generality, in the following, we assume that all computations are performed in the Legendre basis. In this case, recall that $\mathbf{J}_{ij} = (-1)^{i+1} \delta_{ij}$. We first note that the only possible form for a matrix $\mathbf{K}$ satisfying Eq.(\ref{eq:equivalence1}) is
\begin{equation}
    \forall_{i=0}^p \forall_{j=0}^{p-1} : (\mathbf{K})_{ij} = 0.
    \label{eq:K_constraint}
\end{equation}
Let $\mathbf{K} = \mathbf{K}^T$. Then, the only non-zero entry in $\mathbf{K}$ is $(\mathbf{K})_{pp}$ and hence $\mathbf{K}$ is the usual ESFR filtering matrix. This completes the first part of the proof.

Using Eq.(\ref{eq:K_constraint}), direct computations shows that $(\mathbf{M}+\mathbf{K})^{-1}$ is given by
\begin{align}
    &(\mathbf{M} + \mathbf{K})^{-1} \nonumber
    =\\
    &\begin{bmatrix}
        (\mathbf{M}^{-1})_{00} &0 & & \cdots& -(\mathbf{M}^{-1})_{00}(\mathbf{M}^{-1})_{pp}(\mathbf{K})_{00}/(1+(\mathbf{M}^{-1})_{pp}(\mathbf{K})_{pp}) \\
        0&(\mathbf{M}^{-1})_{11}&0 & \cdots& -(\mathbf{M}^{-1})_{11}(\mathbf{M}^{-1})_{pp}(\mathbf{K})_{11}/(1+(\mathbf{M}^{-1})_{pp}(\mathbf{K})_{pp}) \\
        \vdots&\vdots&\vdots && \vdots \\
        0&0&0&\cdots&(\mathbf{M}^{-1})_{pp}/(1+(\mathbf{M}^{-1})_{pp} (\mathbf{K})_{pp})
    \end{bmatrix}
    \label{eq:MK_form}
\end{align}
Let us now assume that $\mathbf{J}\mathbf{Q} = \mathbf{Q}\mathbf{J}$. Then, knowing that $\mathbf{l} = -\mathbf{J}\mathbf{r}$ and that $\mathbf{J} = \mathbf{J}^{-1}$, we must have
\begin{align}
    (\mathbf{M} + \mathbf{Q})\mathbf{J} = \mathbf{J}(\mathbf{M}+\mathbf{Q})
    &\implies
    \mathbf{J}(\mathbf{M} + \mathbf{Q})^{-1} = (\mathbf{M} + \mathbf{Q})^{-1}\mathbf{J} \\
    &\implies
    \mathbf{J}(\mathbf{M} + \mathbf{Q})^{-1}\mathbf{r} = (\mathbf{M} + \mathbf{Q})^{-1}\mathbf{J}\mathbf{r} \\
    &\implies
    \mathbf{J}(\mathbf{M} + \mathbf{Q})^{-1}\mathbf{r} = -(\mathbf{M} + \mathbf{Q})^{-1}\mathbf{l}
\end{align}
Using Eq.(\ref{eq:equivalence2}) and Eq.(\ref{eq:equivalence3}), this implies that
\begin{equation}
    \mathbf{J}(\mathbf{M} + \mathbf{K})^{-1}\mathbf{r} = -(\mathbf{M} + \mathbf{K})^{-1}\mathbf{l}.
    \label{eq:intermediate}
\end{equation}
Substituting Eq.(\ref{eq:MK_form}) in Eq.(\ref{eq:intermediate}), one obtains
\begin{align}
    &\begin{bmatrix}
        (-1)^1(\mathbf{M}^{-1})_{00} +(-1)^1(\mathbf{M}^{-1})_{00}(\mathbf{M}^{-1})_{pp}(\mathbf{K})_{00}/(1+(\mathbf{M}^{-1})_{pp}(\mathbf{K})_{pp}) \\
        (-1)^2(\mathbf{M}^{-1})_{11} +(-1)^3(\mathbf{M}^{-1})_{11}(\mathbf{M}^{-1})_{pp}(\mathbf{K})_{11}/(1+(\mathbf{M}^{-1})_{pp}(\mathbf{K})_{pp}) \\
        (-1)^3(\mathbf{M}^{-1})_{11} +(-1)^p(\mathbf{M}^{-1})_{22}(\mathbf{M}^{-1})_{pp}(\mathbf{K})_{22}/(1+(\mathbf{M}^{-1})_{pp}(\mathbf{K})_{pp}) \\
        \vdots \\
        (-1)^{p+1}(\mathbf{M}^{-1})_{pp}/(1+(\mathbf{M}^{-1})_{pp} (\mathbf{K})_{pp})
    \end{bmatrix}
    =\nonumber \\
    &\begin{bmatrix}
        (-1)^1(\mathbf{M}^{-1})_{00} + (-1)^0(\mathbf{M}^{-1})_{00}(\mathbf{M}^{-1})_{pp}(\mathbf{K})_{00}/(1+(\mathbf{M}^{-1})_{pp}(\mathbf{K})_{pp}) \\
        (-1)^2(\mathbf{M}^{-1})_{11} +(-1)^1(\mathbf{M}^{-1})_{11}(\mathbf{M}^{-1})_{pp}(\mathbf{K})_{11}/(1+(\mathbf{M}^{-1})_{pp}(\mathbf{K})_{pp}) \\
        (-1)^3(\mathbf{M}^{-1})_{22} +(-1)^2(\mathbf{M}^{-1})_{22}(\mathbf{M}^{-1})_{pp}(\mathbf{K})_{11}/(1+(\mathbf{M}^{-1})_{pp}(\mathbf{K})_{pp}) \\
        \vdots \\
        (-1)^{p+1}(\mathbf{M}^{-1})_{pp}/(1+(\mathbf{M}^{-1})_{pp} (\mathbf{K})_{pp})
    \end{bmatrix}
\end{align}
Clearly, this constraint can only be satisfied if $(\mathbf{K})_{pp}$ is the only non-zero entry of $\mathbf{K}$. Hence $\mathbf{K}$ is the ESFR filetring matrix, which completes the second part of the proof.
\end{proof}

It should be noted that the results derived in this section, strictly speaking, hold for non-upwind numerical fluxes. If an upwind numerical flux is used ($\alpha = 0$), then Eq.(\ref{eq:equivalence2}) or Eq.(\ref{eq:equivalence3}) must be removed from the statement of Lemma \ref{lem:equivalence}. Since Eq.(\ref{eq:equivalence1}) ensures $\mathbf{K}$ has $p+1$ free parameters, the remaining constraint can always be satisfied, no matter the choice of $\mathbf{Q}$. Hence, in this specific case, all FR schemes can be expressed as FDG schemes. In practice, this statement is insignificant since it is limited to a narrow class of numerical fluxes and only makes sense in 1D.

\section{Example of Linear Instability for GSFR Schemes}
\label{app2}
Although GSFR schemes were initially introduced as an extended infinite family of linearly stable FR schemes, here we present a simple numerical experiment showing that these schemes are not stable in the usual sense, even when $\forall_{i=1}^p b_i > 0$. To investigate the linear stability of GSFR schemes, their associated system matrix is constructed, \textit{ie.,}, GSFR schemes are applied to the linear advection problem discussed previously and rewritten as a linear ODE system
\begin{equation}
    \frac{d}{dt}
    \begin{bmatrix}
        \hat{\mathbf{u}}^h_0 \\
        \vdots \\
        \hat{\mathbf{u}}^h_{N-1}
    \end{bmatrix}
    = \mathbf{A}_{sys}
    \begin{bmatrix}
        \hat{\mathbf{u}}^h_0 \\
        \vdots \\
        \hat{\mathbf{u}}^h_{N-1}
    \end{bmatrix},
\end{equation}
where $\mathbf{A}_{sys} \in \mathbb{R}^{(p+1)N \times (p+1)N}$ is defined as the system matrix of the scheme. The stability of any given scheme can then be assessed by computing the eigenvalues of $\mathbf{A}_{sys}$. The eigenvalues of $\mathbf{A}_{sys}$ are shown in Figure \ref{fig:B_GSFR} for a GSFR scheme with $p=3$, a pure upwind numerical flux, $b_1 = 0.03$, $b_2 = 0.03$ and $b_3 = 0.0075$. This scheme is used to solve the linear advection equation on $[-1,1]$ subject to periodic boundary conditions with an advection speed given by $a=2$. The number of elements is fixed to $N=10$. As can be seen, it is clear that some of the eigenvalues of the system matrix have a positive real part, which implies the scheme is unstable. The eigenvalue with the maximum real part was found to be $\lambda = 0.148+18.385 I$.
\begin{figure}
    \centering
    \includegraphics[width=0.7\linewidth]{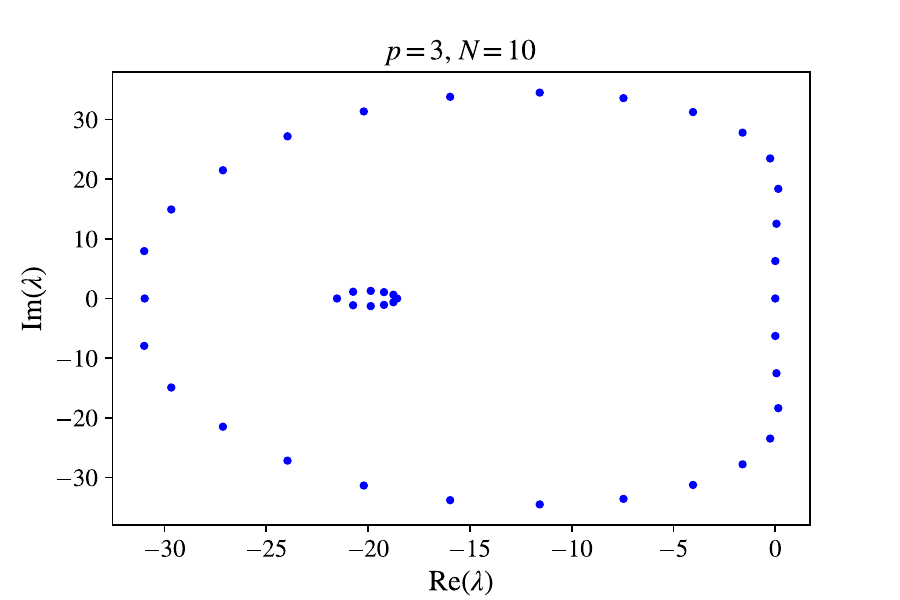}
    \caption{Eigenvalues of the system matrix of a GSFR scheme with $p=3$, $b_1 =\quad 0.03$ $b_2 = 0.03$, $b_3= 7.5 \times 10^{-3}$ and an upwind numerical flux applied to the linear advection equation with $a=2$. 10 elements were used to discretize the interval $[-1,1]$.}
    \label{fig:B_GSFR}
\end{figure}

\bibliographystyle{elsarticle-num} 
\bibliography{cas-refs.bib}

\end{document}